\documentclass[12pt,makeidx]{amsart}

\makeatletter
\let\saved@setaddresses\@setaddresses %
\let\@setaddresses\relax              %
\makeatother

\usepackage{imakeidx}
\makeindex

\usepackage{xcolor} %
\definecolor{darkred}{rgb}{0.5,0,0} %
\definecolor{darkblue}{rgb}{0,0,0.5} %

\usepackage[pagebackref,colorlinks,linkcolor=darkred,citecolor=darkblue,urlcolor=darkblue,hypertexnames=true]{hyperref}

\DeclareMathAlphabet{\lcal}{U}{dutchcal}{m}{n}

\usepackage{amssymb}
\usepackage{bbm}

  \usepackage[scr=boondoxo]{mathalfa}

\usepackage{graphicx}%
\usepackage{yhmath}%
\usepackage{mathdots}%
\usepackage{MnSymbol}%
\usepackage{xfrac}
\usepackage{nicefrac}
\usepackage{euflag}

\usepackage{color}
\usepackage[shortlabels]{enumitem}

\usepackage{comment}

\usepackage{relsize}

\usepackage{url}

\newcommand{\df}[1]{{\it{#1}}{\index{#1}}} %

\usepackage{ulem}

\newtheorem{theorem}            {Theorem}[section]
\newtheorem{corollary}          [theorem]{Corollary}
\newtheorem{proposition}        [theorem]{Proposition}

\newtheorem{lemma}              [theorem]{Lemma}

\newtheorem{example}            [theorem]{Example}
\newtheorem{remark}             [theorem]{Remark}

\newcommand{\cB}{B}
\newcommand{\cH}{\mathcal H}
\newcommand{\cK}{\mathcal{K}}

\newcommand{\cL}{\mathcal{L}}
\newcommand{\cM}{\mathcal{M}}

\newcommand{\cZ}{\mathcal{Z}}

\newcommand{\CC}{\mathbb{C}}
\newcommand{\NN}{\mathbb{N}}
\newcommand{\RR}{\mathbb{R}}

\newcommand{\ZZ}{\mathbb{Z}}
\newcommand{\FF}{\mathbb{F}}

 \newcommand{\sAwM}{\mathscr{A}_{\Ww,M}}

 \newcommand{\sB}{\mathscr{B}}

 \newcommand{\fC}{\mathfrak{C}}

 \newcommand{\fD}{\mathfrak{D}}

 \newcommand{\vg}{{\tt{g}}}

 \newcommand{\real}{\operatorname{real}}

\newcommand{\blangle}{\big \langle}
\newcommand{\brangle}{\big \rangle}

\newcommand{\HE}{\mathcal{E}}
\newcommand{\HF}{\mathcal F}
\newcommand{\he}{\lcal e}

\newcommand{\hf}{\lcal f}

\newcommand{\groupW}{\mathscr W}
\newcommand{\groupWw}{\groupW_{\le \Ww}}
\newcommand{\groupWx}[1]{\groupW_{\le {#1}}}

\newcommand{\groupY}{\mathfrak{Y}}

\newcommand{\vh}{{\tt{h}}}
\newcommand{\llx}{\langle x\rangle}
\newcommand{\llxinv}{\langle x^{-1}\rangle}
\newcommand{\lly}{\mathcal Y}%

\newcommand{\semiyN}[1]{\lly_{\le {#1}}} %
\usepackage{amsmath}
\DeclareMathOperator{\Her}{\ell\text{-}Frac}

\newcommand{\pE}{E}
\newcommand{\pF}{F}

\newcommand{\Ztg}{\mathbb{Z}\strut^{*\vg}_2}
\newcommand{\Zth}{\mathbb{Z}\strut^{*\vh}_2}

\newcommand{\KL}{L}

\newcommand{\msT}[1]{\mathscr{T}_{\mathsmaller{{#1}}}}
\newcommand{\msTK}[1]{\mathscr{T}_{\mathsmaller{{#1},K}}}
\newcommand{\msLK}[2]{\mathscr{L}_{\mathsmaller{{#1},{#2},K}}}
\newcommand{\wfC}[1]{\widehat{\mathfrak{C}}_{#1}}

\newcommand{\PXee}{P}
\newcommand{\freeG}{\FF}
\newcommand{\freesemi}{\llx}

\newcommand{\ztx}{x} %

\newcommand{\Wu}{\mathscr{u}} %
\newcommand{\Wv}{\mathscr{v}} %
\newcommand{\Ww}{\mathscr{w}} %
\newcommand{\Wz}{\mathscr{z}} %
\newcommand{\Ws}{\mathscr{s}} %
\newcommand{\Wt}{\mathscr{t}} %
\newcommand{\Wb}{\mathscr{b}} %
\newcommand{\Wa}{\mathscr{a}} %

\newcommand{\lWu}{\mathscr{u}} %
\newcommand{\Ya}{\lcal{a}} %
\newcommand{\Yb}{\lcal{b}} %
\newcommand{\gYu}{\mathfrak{u}} %

\newcommand{\WGu}{u}  %
\newcommand{\WGv}{v}  %
\newcommand{\YGg}{g} %
\newcommand{\YGh}{h} %
\newcommand{\WYu}{u} %
\newcommand{\WYv}{v} %

\newcommand{\lralpha}{\alpha}%
\newcommand{\lrbeta}{\beta} %
\newcommand{\lru}{\Wu} %
\newcommand{\lrv}{\Wv}  %
\newcommand{\lrmu}{\mu} %
\newcommand{\lrnu}{\nu} %
\newcommand{\lrh}{h}  %
\newcommand{\msB}{\mathscr{B}} %

\newcommand{\mbo}{1}

\usepackage{xr-hyper} %

\title[Noncommutative Fejér--Riesz Factorization]{Fejér--Riesz factorization for positive noncommutative trigonometric polynomials}

\author[I.\ Klep]{Igor Klep${}^{1,Q}$}
\address{Igor Klep, Faculty of Mathematics and Physics, University of Ljubljana 
\& Famnit, University of Primorska, Koper 
\& Institute of Mathematics, Physics and Mechanics,
Ljubljana, Slovenia}
\email{igor.klep@fmf.uni-lj.si}
\thanks{${}^1$Supported by the Slovenian Research Agency 
	program P1-0222 and grants J1-50002,
	N1-0217,
J1-60011, J1-50001, J1-3004 and J1-60025. Partially supported by the Fondation de l’Ecole polytechnique as part
of the Gaspard Monge Visiting Professor Program. IK thanks
Ecole Polytechnique and Inria for
hospitality during the preparation of this manuscript.}

\author[J. Levenson]{Jacob Levenson}
\address{Jacob Levenson, Department of Mathematics\\
	University of Florida\\ Gainesville } %
\email{levenson.j@ufl.edu}

\author[S. McCullough]{Scott McCullough}
\address{Scott McCullough, Department of Mathematics\\
	University of Florida\\ Gainesville} %
    
\email{sam@math.ufl.edu}

\thanks{${}^Q$This work was performed within the project COMPUTE, funded within the QuantERA II 
Programme that has received funding from the EU's H2020 research and innovation programme under the GA No 101017733 {\normalsize\euflag}}



\numberwithin{equation}{section}

\makeindex

\begin{document}
\baselineskip=18pt

\begin{abstract}
We prove a Fej\'er--Riesz type factorization for positive matrix-valued noncommutative trigonometric polynomials on $\groupW\times\groupY$, where $\groupW$ is either the free
semigroup $\llx_\vg$ or the free product group $\Ztg$, and $\groupY$ is a  finitely generated  group. 
More precisely, using the shortlex order, 
if $A$ has degree at most $\Ww$ in the $\groupW$ variables and is strictly positive on all unitary representations of $\groupW\times\groupY$, then $A=B^{*}B$ with $B$ analytic and of $\groupW$-degree at most $\Ww$; this degree bound is optimal, and strict positivity is essential. As an application, 
we obtain degree-bounded sums-of-squares certificates for Bell-type inequalities in $\CC[\Ztg\times \Zth]$ from quantum information theory.

In the special case $\groupY=\ZZ^{\vh}$ we recover, in the matrix-valued setting, the classical commutative multivariable Fej\'er--Riesz factorization. For trivial $\groupY$ we obtain a ``perfect'' group-algebra Positivstellensatz on $\Ztg$ that does not require strict positivity; 
this result is sharp in the sense that no such perfect degree bound can hold on $\ZZ_2 \ast\ZZ_3$ and $\ZZ_3^{*2}$
as demonstrated by counterexamples.

To establish our main results, two novel ingredients of independent interest are developed: (a) a positive semidefinite Parrott theorem with entries given by functions on a group; and (b) solutions to positive semidefinite matrix completion problems for $\llx_\vg$ or  $\Ztg$ indexed by words %
of length $\le \Ww$.
\end{abstract}

\maketitle

\section{Introduction}
The classical Fejér--Riesz theorem asserts that a univariate trigonometric polynomial that is positive on the unit circle factors as the modulus square of an analytic polynomial. Due to its importance across many disciplines (e.g., minimum-phase filter design, prediction theory, and the analysis of Toeplitz operators and moment problems) it has been generalized in many directions: to matrix- and operator-valued polynomials \cite{Ro68,DR10}, and to multivariate settings \cite{GW05,Dritschel,DW05}. The Fejér--Riesz theorem is also a particular instance of a Positivstellensatz \cite{smu91,Pu93}, a pillar of real algebraic geometry \cite{BCR98,Ma08,Sc24}. 

During this century, 
motivated by linear systems theory \cite{SIG98,dOHMP09}, quantum physics \cite{brunner,KMP22} and free probability \cite{MS17,Gu16}, 
a noncommutative version of function theory \cite{KVV14,MuSo11,AM15,BMV16,PTD22}, also known as free analysis \cite{Vo10}, is under development. 
This program includes noncommutative Fejér--Riesz–type factorizations and nc Positivstellens\"atze. See for instance \cite{Mc01,He02,HM04,HMP04,Po95,JM12,JMS21} and the references cited therein.
We also note inherent limitations on algorithmic detection of noncommutative positivity: in certain tensor-product settings positivity is undecidable \cite{MSZ}.

In this paper, we prove a Fejér–Riesz factorization for positive, matrix-valued \textit{noncommutative trigonometric} polynomials 
$A$ 
on a mixed domain $\groupW\times \groupY$, where $\groupW$ is either a free semigroup or a free product of copies of $\ZZ_2$ and $\groupY$ is a
finitely generated group.
Along the way we develop two tools of independent interest—a positive semidefinite (psd) Parrott theorem with entries given by functions on a group, and degree-controlled psd matrix completions \cite{GJSW, BW, CJRW, GW05, GKW, sparsity, 2-chordal} indexed by words. 
To explain our contribution in more detail, we require some notation.

\subsection{Notation}
  Fix a positive integer $\vg$ and let $\freeG_\vg$\index{$\mathbb F_{{\tt{g}}}$}\index{free group} denote the free group on the alphabet
  $x=\{x_1,\dots,x_\vg\}.$ The group $\freeG_\vg$  contains the free semigroup $\llx_\vg=\langle x_1,\dots,x_\vg\rangle$\index{free semigroup}\index{$\langle x_1,\dots,x_{\tt{g}}\rangle$}\index{$\langle x\rangle_{\tt g}$}
  generated by $x$ as well as the free semigroup $\llxinv_\vg$ generated by $x^{-1}=\{x_1^{-1},\dots,x_\vg^{-1}\}$.
  The set of \df{left fractions} \index{$\ell$-Frac}
\[
 \Her \llx_\vg  = \{\Wu^{-1} \Wv: \Wu,\Wv \in \freesemi_\vg\}
\subset\freeG_\vg
\]
 plays an important role. 
  
  Let $\Ztg$\index{$\mathbb Z_2^{*{\tt g}}$} denote the free product 
  of $\ZZ_2$  with itself $\vg$ times. Thus $\Ztg$ is the group
  with generators $\ztx=\{\ztx_1,\dots,\ztx_\vg\}$ and relations $\ztx_j^2=e,$
  where $e$ is the identity of $\Ztg.$  Note 
  that $\ztx_j^{-1}=\ztx_j$ for  $1 \leq j \leq \vg$.

  Let $\groupY$\index{$\mathfrak Y$} denote a group
with a finite set $y=\{y_1,\dots,y_\vh\}$ such that 
\begin{equation}\label{eq:groupY}
 \groupY = \Her\lly =\{\Ya^{-1}\Yb : \Ya,\Yb \in \lly\},
\end{equation}
 where $\lly$\index{$\mathcal Y$} denotes the semigroup generated by $y$.
 For instance, the additive 
 group $\ZZ^\vh$ and $\Zth$ both have this property,
 in the first
 case with {each}  generator $y_j\in \ZZ^{\vh}$ being the element with 
 a $1$ in the $j$-th position and a $0$ elsewhere and in the second
 case {with each generator $y_j \in \Zth $
 being one of the usual generators of $\Zth$ 
 (meaning $y_j$ does not commute with $y_i$ for $i \neq j$ and $y_j^2 =e$).} In general, if $\groupY$ is finitely generated 
 $y$ can always be taken to be a set of generators for $\groupY.$
 
 Let $\groupW$\index{$\mathscr W$} denote either $\Ztg$ or $\llx_\vg.$ 
 Elements of the direct product $\Her \groupW \times \groupY$ have the form 
  $\Wu \, \Ya$ for $\Wu \in \Her \groupW$ and $\Ya \in \groupY.$
  Note that $\Ztg$ and $\groupY$ are naturally subgroups of {$\Ztg \times \groupY$}.
  Likewise $\groupY$ is naturally a subgroup of $\Her \llx_\vg\times \groupY;$ whereas
  $\Her \llx_\vg$ is naturally left $\langle x^{-1}\rangle_\vg$-invariant and right $\llx_\vg$-invariant.

  A \df{unitary representation} $\pi$ of $\Ztg \times \groupY$  on a Hilbert space  $\HF$ is given by 
 a unitary representation $\rho$ of $\groupY$ on $\HF$ and  a tuple of $(U_1,\dots,U_\vg)$
  of unitary operators on $\HF$ such that   $U_j^2=I_{\HF}$ and  $U_j$ commutes with $\rho$ for each $j$
  (meaning $U_j \rho(\gYu) =\rho(\gYu)U_j$ for all $\gYu\in\groupY$), with 
\begin{equation}
\label{e:piUrho0}
 \pi(\ztx_{i_1}\cdots \ztx_{i_k} \, \gYu) = U_{i_1} \cdots U_{i_k} \, \rho(\gYu).
\end{equation}

   For the purposes here, a unitary representation $\pi$ of 
$\llx_\vg \times \groupY$ or    
   $\Her\llx_\vg \times \groupY$ on 
 a Hilbert space $\HF$ 
 is a unitary representation of $\freeG_\vg\times\groupY$ on $\HF$; that is, 
 $\pi$  is given by a unitary representation $\rho$ of $\groupY$ on $\HF$ and 
 a tuple of unitaries $U=(U_1,\dots,U_\vg)$ acting on $\HF$ that commute with $\rho$
 via
\begin{equation}
\label{e:piUrho}
 \pi(\Wu^{-1}\Wv \, \gYu) = (U^\Wu)^*\,  U^\Wv  \rho(\gYu),
\end{equation}
 where $U^\Wv$ is defined in the canonical fashion.
If we are considering $\llx_\vg \times \groupY$, then $\Wu^{-1}$ in \eqref{e:piUrho} is the empty word.

We order
  the reduced words in $\groupW$ (either $\llx_\vg$ or $\Ztg$) by
  the \df{shortlex order} (length and then dictionary).
The shortlex order is a well-ordering; it is a
 total ordering on $\groupW$ with the property that every non-empty subset of $\groupW$ has a least element.
  Given a word $\Ww\in\groupW$,  let $\groupWw=\{\Wv\in\groupW: \Wv \le \Ww\}$ and 
\begin{equation}
    \label{e:l-frac}
\Her \groupWw = \{\Wu^{-1}\Wv : \Wu,\Wv\in \groupWw\}.
\end{equation}
  For the group $\groupY$ as in \eqref{eq:groupY}, let  $\semiyN{M}$ denote  those words in $y$ of
  length at most $M.$

Given $\Ww\in\groupW,$  positive integers $K$ and $M$ and
$A_g\in M_K(\CC)$ for each $g\in \Her \groupWw\times \Her \semiyN{M},$ 
the expression,
\begin{equation}
  \label{eq:trigPoly}
  A = \sum\{ A_g\,  g :  g \in \Her \groupWw\times\Her \semiyN{M} \} %
\end{equation}
is an analog of an $M_K(\CC)$-valued \df{trigonometric polynomial},
or simply \df{polynomial} for short, whose \df{bidegree} is at most $(\Ww,M).$
Similarly,
now with also $K^\prime,M^\prime$ as positive integers, 
given $B_{g}\in M_{K^\prime,K}$ for $g \in \groupWw \times \semiyN{M^\prime},$ the expression
\begin{equation}\label{eq:analPoly}
  B=\sum \{ B_g \, g : g  \in \groupWw \times \semiyN{M'} \}
\end{equation}
is the analog of an \df{analytic polynomial}. In this case,  the \df{bidegree} of $B$ is at most  $(\Ww,M^\prime).$
 For $B$ as in equation~\eqref{eq:analPoly}, let
\[
 B^* =  \sum_{\lrh \in \groupWw\times \semiyN{M'}} B_\lrh^*  \lrh^{-1}
\]
 and interpret $B^*B$ in the canonical way.  In particular, $B^*B$ is a
 $M_K(\CC)$-valued trigonometric polynomial of bidegree  at most $(\Ww,M^\prime).$

Given a unitary representation $\pi$ of %
$\Her\groupW\times \groupY$,
the
{\it value} of $A$ 
as in 
\eqref{eq:trigPoly}
at $\pi$ is
\begin{equation}\label{eq:evalPoly}
 A(\pi) = \sum\{ A_g \otimes \pi(g) :  g \in \Her \groupWw\times\Her\semiyN{M} \}. %
\end{equation}
 The evaluation $B(\pi)$ of an analytic polynomial $B$ at $\pi$ is defined similarly. 
 In particular,
\[
 B^*B(\pi) = B(\pi)^* B(\pi).
\]

\subsection{Main results}
 In this section we state our main results and provide some context. 
 An operator $T$ on a complex Hilbert space $\mathcal{G}$ is \df{positive definite}, abbreviated
 \df{pd} and denoted  \df{$T\succ 0$}, if $\langle Tg,g\rangle >0$ for all $0\ne g\in \mathcal{G}.$
 An operator $T$ is \df{positive semidefinite} (\df{psd}), denoted \df{$T\succeq 0$}, if $\langle Tg,g\rangle \geq0$ for all $g\in \mathcal{G}.$  Recall the definition of $\Her \groupWw$
 from equation~\eqref{e:l-frac}. In particular, the words $\Wu,\Wv$ come from  either  the group $\Ztg$ or the semigroup $\freesemi_\vg$
 as the case may be.

  \begin{theorem}[Noncommutative Fejér--Riesz theorem]
\label{t:main}
Let $\groupW$ denote either the free semigroup $\freesemi_\vg$ 
or the free product group $\Ztg$. 
    Fix $\Ww\in \groupW$ and  positive integers $M$ and $K$ 
    and let $A$ denote a given $M_K(\CC)$-valued trigonometric polynomial of bidegree at most $(\Ww,M).$
    If 
\[
   A(\pi)\succ 0
\]
    for each separable  Hilbert space $\HF$
    and  unitary representation $\pi:\Her\groupW\times\groupY \to \cB(\HF),$
then  there exist positive integers $K^\prime$ and $M^\prime$ and
   an $M_{K^\prime,K}(\CC)$-valued analytic polynomial $B$ of bidegree at most $(\Ww,M^\prime)$
   such that 
\begin{equation}\label{e:main1}
 A= B^*B.\index{Fejér--Riesz factorization}
\end{equation}
\end{theorem}

The proof of Theorem \ref{t:main} is given in Section \ref{sec:proofs}, see Theorem \ref{t:main:again}.

\begin{remark}\rm
\label{r:t:main:sos}\index{Positivstellensatz}
\pushQED{\qed} %
 The conclusion \eqref{e:main1} of Theorem~\ref{t:main}, 
 as well as that of  Corollary~\ref{c:main:noX} and
 Theorem~\ref{t:main:noY} below, can be interpreted as a sum-of-squares
  (sos) Positivstellensatz-type certificate. Namely, 
there exist positive integers $M^\prime$ and $r$ 
and $M_K(\CC)$-valued analytic
polynomials $B_j$ of bidegree at most $(\Ww,M')$ such that
\[
  A= \sum_{j=1}^r B_j^*B_j. \qedhere
\]
\end{remark}

\begin{remark}\rm
Because our groups are at most countable, it suffices to consider representations on separable Hilbert space 
for our results; see, e.g., 
Theorem \ref{t:main}. %

Let $G$ be a countable group and let $\cH$ be a Hilbert space. 
Write $\cB(\cH)[G]=\cB(\cH)\otimes \CC[G]$ for the (algebraic) tensor product, so every $A\in \cB(\cH)[G]$ has the form
\[
A=\sum_{g\in F} T_g\otimes g
\quad (F\subseteq G\text{ finite},\ T_g\in \cB(\cH)).
\]
Suppose there exists a Hilbert space $\cK$ and a unitary representation $\pi:\CC[G]\to \cB(\cK)$ such that
\[
(I\otimes \pi)(A)=\sum_{g\in F} T_g\otimes \pi(g)\in B(\cH\otimes \cK)
\]
is not positive semidefinite. Then there exists a {separable} Hilbert space $\cK_0$ and a unitary representation
$\pi_0:\CC[G]\to  \cB(\cK_0)$ such that $(I\otimes \pi_0)(A)$ is not positive semidefinite on $\cH\otimes \cK_0$.

Indeed, suppose $(I\otimes \pi)(A)$ fails to be positive semidefinite on $\cH\otimes \cK$. Hence there exists
$v\in \cH\otimes \cK$ with
\[
\langle (I\otimes \pi)(A)v,v\rangle<0.
\]
The quadratic form $q(x)=\langle (I\otimes \pi)(A)x,x\rangle$ is continuous, and the algebraic tensor product
$\cH\odot \cK$ is dense in $\cH\otimes \cK$. Therefore we may choose
a finite sum
\[
\check v=\sum_{j=1}^m h_j\otimes k_j \in \cH\odot \cK
\]
such that $\langle (I\otimes \pi)(A)\check v,\check v\rangle<0$.

Define
\[
\cK_0 = \overline{\operatorname{span}}\{ \pi(g)k_j\, \colon\, g\in G,\ 1\le j\le m \} \subseteq \cK.
\]
Because $G$ is countable and there are only finitely many $k_j$, the space $\cK_0$ is separable. It is invariant under
$\pi(G)$, so the restriction $\pi_0:=\pi|_{\cK_0}$ is a unitary representation of $G$ on $\cK_0$. Clearly $\check v\in \cH\otimes \cK_0$.
Moreover,
\[
\langle (I\otimes \pi_0)(A)\check v,\check v\rangle
=\langle (I\otimes \pi)(A)\check v,\check v\rangle<0.
\]
Hence $(I\otimes \pi_0)(A)$ is not positive semidefinite on $\cH\otimes \cK_0$, as claimed.\qed
\end{remark}

\begin{remark}\rm
\label{r:t:main:bounds}
 While there is no degree bound for $B$ in \eqref{e:main1} relative to the $y$ variables,
 the degree bound relative to the %
 {variables of $\groupW$}
 is best possible.  
The conclusion of Theorem \ref{t:main} without any degree bounds or analyticity on $B$
 is due to
Helton-McCullough \cite{HM04}.

  The proof of Theorem~\ref{t:main} gives a bit more than stated. 
  It turns out that if $A$ satisfies the hypotheses of Theorem~\ref{t:main}, then
  there is an $\epsilon>0$ such that for each separable  Hilbert space $\HF$
    and  unitary representation $\pi:\Her\groupW\times\groupY \to \cB(\HF)$ 
    the inequality
\[ 
  A(\pi)  \succeq \epsilon (I_K\otimes I_{\HF})
\]
   holds. See Proposition~\ref{no-eps-no-problem}.  Moreover, with the normalization $A_{e,e}=I_K,$ for a fixed $\epsilon,$ 
   there is an $M^\prime$
  that depends only on $\epsilon,K,M$ and $|\groupWw|,$ that suffices for the conclusion of 
  Theorem~\ref{t:main}, though we do  not have a concrete
  estimate for the size of $M^\prime.$  See Theorem~\ref{t:main:again}.
  { We also note that, even if $\groupY$ were not finitely generated,
  a given polynomial $A$ depends on only a finitely generated subgroup
  of $\groupY.$}
  \qed
\end{remark}

\begin{remark}\rm
The choice of  $\groupY$ equal to the direct product $\ZZ^\vh$ of $\ZZ$ with itself $\vh$ times 
in Theorem~\ref{t:main} is
of particular interest. Since  irreducible unitary representations of the group $\ZZ^\vh$
 are one-dimensional and determined by points in the torus, it suffices to assume
 that $A(\pi)\succ 0$ for representations $\pi$ determined by 
 unitaries $U=(U_1,\dots,U_\vg)$ along with tuples $(\zeta_1,\dots,\zeta_\vh)$ 
 satisfying $|\zeta_j|=1.$  Moreover, in this case, choosing  $\groupW =\NN$ and thus $\Her\groupW = \ZZ,$
 reduces to the case of (classical) matrix-valued trigonometric polynomials in several (commuting) variables
 leading to the factorization result in \cite{Dritschel} at least in the matrix-valued case.
 \qed
\end{remark}

 In the case that $\groupY$ is not present (equivalently it is the trivial group) in Theorem \ref{t:main}, 
 the strict positivity hypothesis is not needed, a fact that
 appears in \cite{Mc01} for 
 the free semigroup $\llx_\vg.$
 (See also \cite{HMP04,BT07,NT13,KVV17,Oz13} for a similar result for the free group $\freeG_\vg$.)
   An analogous result holds for the case of $\Ztg$:

\begin{theorem}[Group algebra Positivstellensatz for $\Ztg$]
 \label{t:main:noY}
   Set $\groupW=\Ztg.$ 
   Let $\Ww\in \groupW,$ a Hilbert space $\HE$ and $A_\lWu\in \cB(\HE)$ for $\lWu\in \Her \groupWw$ be given
   and let 
\[
  A=\sum_{\lWu\in \Her \groupWw} { A_\lWu \lWu}
\]
 denote the resulting trigonometric polynomial. Thus the degree of $A$ is at most $\Ww.$  If $A(\pi)\succeq 0$
 for all unitary representations $\pi$ of $\Ztg$ on separable Hilbert space, then there is an auxiliary
 Hilbert space $\HE^\prime$ and an analytic
 polynomial $B$ of degree at most $\Ww$ with coefficients in $B(\HE,\HE^\prime)$ such that 
\[
 A= B^*B.
\]
\end{theorem}

\begin{remark}\rm
\label{r:main:Y}
   Theorem~\ref{t:main:noY} holds in the following form when $\groupY$ is a finite group 
   of cardinality $M$ for $\groupW$ either $\llx_\vg$ or $\Ztg.$ Namely, given $A$ as in equation~\eqref{eq:evalPoly} 
   with coefficients in $\cB(\HE)$ of bidegree $(w,M),$
   if $A(\pi)\succeq 0$ for all unitary representations $\pi$ of $\Her \groupW \times \groupY$ on separable Hilbert space,
   then there  exists an auxiliary Hilbert space $\HE^\prime$ and an
   analytic $\cB(\HE,\HE^\prime)$  polynomial $B$ of bidegree $(w,M)$ such that $A=B^*B.$

   {In this setting, while not carried out here, a careful analysis of the matrix completion argument given in the proofs
   shows that it suffices to consider unitary representations $\pi$ on Hilbert space of dimension at most
   $MK\, |\groupWw|.$} 
\end{remark}

\begin{remark}\rm
Bell inequalities\index{Bell inequality} are pillars of quantum physics and quantum information theory.
They were introduced in the seminal paper \cite{Be64} and have been instrumental to experimentally demonstrate \cite{nobel} the validity of quantum mechanics. 
Violation of a Bell inequality serves as an indicator of entanglement and implies that a physical
interaction cannot be explained by any classical picture of physics \cite{brunner}. 
Mathematically a Bell inequality is simply a special type of inequality on trigonometric polynomials in the group algebra $\CC[\Ztg\times\Zth]$. The simplest example is the Clauser-Horne-Shimony-Holt (CHSH) inequality \cite{CHSH69}, where $\vg=\vh=2$, and letting $x_1,x_2$ and $y_1,y_2$ denote the generators of $\Ztg$ and $\Zth$, respectively, we have
\[
\text{CHSH}:\quad  x_1y_1+x_1y_2+x_2y_1-x_2y_2 \leq 2\sqrt2.
\]
Theorem \ref{t:main} provides a new degree-bounded Positivstellensatz-type certificate to validate Bell inequalities with matrix coefficients such as those appearing in quantum steering, cf.~\cite[Equation (12)]{steering}.\qed
\end{remark}

\begin{remark}\rm
The strict positivity assumption in Theorem \ref{t:main} is needed as the conclusion can fail if $A(\pi)$ is merely positive semidefinite at all unitary representations $\pi$. 
This is shown for the case of Bell inequalities in
$\CC[\Ztg\times\Zth]$ in \cite[Theorem 1.3]{FKMPRSZ+}.
Related phenomena occur beyond Bell settings: positivity in tensor products of free algebras is undecidable \cite{MSZ}; this underscores the necessity of  strict positivity in our factorization results.
\qed
\end{remark}

\begin{remark}\rm
  \label{r:noZm}
Theorem~\ref{t:main:noY} is special to $\Ztg$ in the sense that its
\textit{perfect} degree bound $\deg B \le \deg A$ need not hold for other free
products of finite cyclic groups.  In particular, Examples~\ref{eg:noZm} and \ref{e:noZ32}
produce trigonometric polynomials $A \in \CC[\ZZ_2 * \ZZ_3]$ and
$A \in \CC[\ZZ_3^{*2}]$, respectively, for which $A(\pi)\succeq 0$ for all unitary
representations $\pi$, but there is no factorization $A = B^* B$ with
$B$ analytic of degree at most $\deg A$.  Our examples do {not}
exclude the possibility that some weaker (non-optimal) degree bound
 might hold for these groups.  
\qed
\end{remark}

We note that in the case that $\groupW$ is absent, 
Theorem~\ref{t:main} produces the following Positivstellensatz, that, 
without the analyticity condition on $B$,  is a consequence of \cite{HM04}.

\begin{corollary}
 \label{c:main:noX}
Consider an $M_K(\CC)$-valued trigonometric polynomial in $y$ variables only,
\[A=\sum_{\gYu \in\Her\lly_{\leq M}} A_{\gYu}\, \gYu.
\]
    If %
    for each separable Hilbert space $\HF$
    and  unitary representation $\pi:\groupY \to \cB(\HF)$ 
    the inequality
$
A(\pi)  \succ 0
$
   holds, then  there exist positive integers $K^\prime$ and $M^\prime$ and
   an $M_{K^\prime,K}(\CC)$-valued analytic polynomial 
\[
    B=\sum_{\Ya \in\lly_{\leq M'}} B_{\Ya} \Ya
\]
   such that 
\[
 A= B^*B.
 \]
\end{corollary}

\subsection{Further results and guide to the paper}
 The claim regarding the existence of $\epsilon>0$ made in Remark~\ref{r:t:main:bounds} is established
 in Section~\ref{s:no-eps-no-problem}.
 Two ingredients of independent interest underlie the proofs of Theorems~\ref{t:main} and \ref{t:main:noY}. 
First, we establish a positive semidefinite (psd) Parrott theorem \cite{parrott} whose (block) entries are functions on a group $G$.
 See Theorem~\ref{p:PsD-Parrot-for-forms}.  This variant of 
 Parrott's theorem then feeds into solving psd matrix completion problems
 for the free semigroup $\llx_\vg$ and  the free product group $\Ztg.$ See Theorem~\ref{l:compHX} 
 and Theorem~\ref{l:completeZ2}, respectively. For $\Ztg$ the solution to the
 matrix completion problem is novel even for the case where $G$ is trivial.

 Section \ref{sec:psd_functions_on_groups_arise_from_unitary_representations}
 provides a bridge to representation theory. The fact that a psd function (kernel) on a group
 is realized as a compression of a unitary representation of that group applies
 to $\Ztg\times G$ and also extends to $\llx_{\vg}\times G,$
 see Proposition \ref{p:freeY}.   Section \ref{s:nested-sequences} 
packages the truncated Gram data into a compact, nested family and extracts a single truncation level depending only on a fixed compact collection of trigonometric polynomials. A precise statement appears as Lemma~\ref{l:nested:tilde} and it is this uniformity that produces degree bounds $M^\prime$ 
alluded to in Remark~\ref{r:t:main:bounds}.  
 Section \ref{sec:proofs} assembles these pieces to prove the Fej\'er–Riesz factorization Theorem \ref{t:main} with optimal $\groupW$–degree under uniform strict positivity and, for trivial $\groupY$, the ``perfect'' group–algebra Positivstellensatz Theorem \ref{t:main:noY} on $\Ztg$.

  Finally, Section \ref{s:examples} 
presents (counter)examples in the cases of  $\ZZ_2*\ZZ_3$ and $\ZZ_3^{*2}$
 that show our results are sharp.

\section{From positive definite to bigger than \texorpdfstring{$\varepsilon$}{e}} 
\label{s:no-eps-no-problem}
The claim made in Remark~\ref{r:t:main:bounds} is established in this section.

\begin{proposition}
\label{no-eps-no-problem}
Let $G$ be a countable group and let 
$A \in M_n(\CC)[G] = M_n(\CC)\otimes \CC[G]$.
If for every Hilbert space $\cK$ and every unitary representation
$\pi:\CC[G]\to \cB(\cK)$ the operator
\[
(I\otimes \pi)(A) \in \cB(\CC^n\otimes \cK)\cong M_n(\CC)\otimes \cB(\cK)
\]
is positive definite, then 
there exists $\varepsilon>0$ such that
\[
(I\otimes \pi)(A) \succeq \varepsilon\, I_{\CC^n\otimes \cK}
\]
for all $\cK$ and all unitary representations $\pi$ on $\cK$.
\end{proposition}

\begin{proof}
View $A\in M_n(\CC)[G]$ 
as an element of $M_n(\CC)\otimes C^*(G)$,
where $C^*(G)$ denotes the full (or universal) group $C^*$-algebra \cite[Definition 9.12.4]{Da25}.
By hypothesis, for every unitary representation $\pi$ of $G$ on $\cK$ the operator
$(\mathrm{id}_{M_n}\otimes \pi)(A)$ is strictly positive on $\CC^n\otimes \cK$.

Let $S$ be the state space of $M_n(\CC)\otimes C^*(G)$. By the GNS construction,
each $\varphi\in S$ is of the form
\[
\varphi(B)=\langle \rho(B)\xi,\xi\rangle
\]
for some representation $\rho$ of $M_n(\CC)\otimes C^*(G)$ on
a Hilbert space $\cL$ and some unit vector $\xi\in \cL$.

Lemma~\ref{l:folklore} below  
is %
well known. It shows that
$\rho$ is unitarily equivalent to a representation of the form
$\mathrm{id}_{M_n(\CC)}\otimes \sigma$ for a suitable  representation
$\sigma:C^*(G)\to B(\cL_0)$ on some Hilbert space $\cL_0$. Consequently,
\[
\varphi(A)=\langle (\mathrm{id}_{M_n(\CC)}\otimes \sigma)(A)\eta,\eta\rangle
\]
for some unit vector $\eta\in \CC^n\otimes \cL_0$. By the assumption of
strict positivity for all unitary representations, the operator
$(\mathrm{id}_{M_n(\CC)}\otimes \sigma)(A)$ is strictly positive, hence
$\varphi(A)>0$ for every $\varphi\in S$.  

The map $S\to\RR$, defined by $\varphi\mapsto \varphi(A)$, is continuous and $S$ is
weak$^*$-compact; therefore this map attains its minimum at some $\varphi_0\in S$.
Since $\varphi(A)>0$ for all $\varphi$, this minimum is a positive number:
set $\epsilon:=\min_{\varphi\in S}\varphi(A)>0$. Thus for all states $\varphi\in S,$
\begin{equation}
    \label{e:no-eps}
     \varphi(A-\epsilon 1)=\varphi(A)-\epsilon\ \ge\ 0. 
\end{equation}

In a $C^*$-algebra, order is separated by states: for self-adjoint $x$,
one has $x\succeq 0$ if and only if  $\varphi(x)\ge 0$ for all states $\varphi$.
Applying  this fact to the self-adjoint element $A-\epsilon\,1$ it follows from 
equation~\eqref{e:no-eps} that $A-\epsilon \, 1 \succeq 0.$
Finally, for any unitary representation
$\pi:\CC[G]\to \cB(\cK)$,
\[
(\mathrm{id}_{M_n(\CC)}\otimes \pi)(A)\ \ge\ \epsilon I_{n}\otimes I_\cK,
\]
as desired.
\end{proof}

\begin{lemma}
\label{l:folklore}
Let $A$ be a $C^*$-algebra. If $\rho:M_n(\CC)\otimes A\to \cB(\cL)$ is a
representation, then there exist a Hilbert space $\cL_0$ and a
representation $\sigma:A\to \cB(\cL_0)$ such that, after a unitary
identification $\cL\cong \CC^n\otimes \cL_0$,
\[
\rho(X\otimes a) = X\otimes \sigma(a)
\]
for all 
$X\in M_n(\CC)$, $a\in A$.
\end{lemma}

\begin{proof}
The restriction $\rho|_{M_n(\CC)\otimes 1}$ is a representation of the
finite-dimensional algebra $M_n(\CC)$, hence unitarily equivalent to
$X\mapsto X\otimes I_{\cL_0}$ on $\CC^n\otimes \cL_0$ for a suitable
Hilbert space $\cL_0$ (cf.~\cite[p.~20, Corollary 1]{Ar76}).
Under this identification, $\rho(1\otimes a)$ commutes with $M_n(\CC)\otimes I_{\cL_0}$,
so 
it must be of the form $I_n\otimes \sigma(a)$
for a  representation $\sigma:A\to \cB(\cL_0)$
(cf.~proof of the von Neumann Double Commutant Theorem in \cite[Theorem 9.4.1]{Da25}).
Hence 
\[\rho(X\otimes a) = \rho(X\otimes 1)\rho(1\otimes a)
= (X\otimes I)(I\otimes \sigma(a))=X\otimes \sigma(a).
\qedhere
\]
\end{proof}

\section{The Parrott theorem for psd functions on a group}\label{sec:parrott}
This section contains the statement and proof 
 of the well-known psd version of the Parrott theorem, \cite{parrott}, {that is of independent interest.} In that paper
it is shown that if the given entries of the relevant block matrices come from
a von Neumann algebra, then there is a solution to the completion problem
that comes from the von Neumann algebra. Here we show that if the entries
are functions on a group, then the solution can be chosen to also be a function 
on that group.  %

Let $G$ denote a group and fix a Hilbert space $\HE.$
 Given a function $p:G\to B(\HE),$
  let $\Upsilon_p$\index{$\Upsilon_p$} denote the associated block matrix,
\begin{equation}
 \Upsilon_p = \begin{pmatrix} p(g^{-1}h) \end{pmatrix}_{g,h\in G}.
\end{equation}
  Let 
\begin{equation}\label{eq:funnyF}
 \mathscr{F} = C_{00}(G,\HE)=\{\phi:G\to \HE, \,  |\{g:\phi(g)\ne 0\}|<\infty \}.
\end{equation}
 Thus $\mathscr{F}$ consists of those $\HE$-valued functions $\phi$ on $G$ 
 such that  $\phi(g)=0$  for all but finitely many $g\in G.$

  The matrix $\Upsilon_p$ determines a sesquilinear \df{form} on $\mathscr{F}$ by
\begin{equation}
    \label{e:Yp}
   \langle \Upsilon_p \varphi,\psi\rangle = \sum_{g,h\in G} \langle p(g^{-1}h) \varphi(h),\psi(g)\rangle
\end{equation}
  for $\varphi,\psi\in\mathscr{F}.$   The form    $\Upsilon_p$ is \df{positive definite}  (pd) \index{pd}
  (resp. \df{positive semidefinite} (psd)) if 
\[
  \langle \Upsilon_p \varphi,\varphi \rangle >0 
\]
  (resp. $\langle \Upsilon_p \varphi,\varphi \rangle \ge 0$)  
 whenever $0\ne \varphi \in \mathscr{F}.$  In particular, $p(g^{-1}h) = p(h^{-1}g)^* \in B(\HE)$
 for each $g,h\in G.$ 

 An element $\lralpha\in G$ 
  induces a bijective  linear map $L_{\lralpha}:\mathscr{F}\to\mathscr{F}$ defined by 
\begin{equation}
    \label{e:Lalpha}
L_\lralpha \varphi(g)= \varphi(\lralpha^{-1}g),
\end{equation} 
 for $\varphi\in\mathscr{F}.$ This map $L_\lralpha$
  is unitary with respect to the form induced by $\Upsilon_p$  
  since, by equation~\eqref{e:Yp}, 
\begin{equation}
    \label{e:YpLx}   
    \begin{split}
 \langle \Upsilon_p  L_{\lralpha} \varphi, L_\lralpha \psi \rangle 
   & = \sum_{g,h\in G} \big \langle  p(g^{-1}h) \varphi(\lralpha^{-1}h),\psi(\lralpha^{-1}g)\big \rangle
\\[5pt] & =  \sum_{g,h\in G} \big \langle p(\lralpha^{-1}g)^{-1} (\lralpha^{-1} h) \varphi(\lralpha^{-1}h),\psi(\lralpha^{-1}g) \big \rangle
\\[5pt]  & = \langle \Upsilon_p\varphi,\psi\rangle.
\end{split}
\end{equation}
 Fix a positive integer $N$, let  $J_N=\{0,1,\dots,N\}$, and fix
 functions $p_{j,k}:G\to B(\HE)$ for $(j,k)\in J_N\times J_N \setminus \{(0,N),(N,0)\}.$ Let
  $\Upsilon_{j,k}=\Upsilon_{p_{j,k}}$ denote the matrix associated to $p_{j,k}.$
   We assume that $\Upsilon_{k,j}$ is the formal adjoint to $\Upsilon_{j,k}$ in the sense that
\[
  \langle \Upsilon_{k,j} \varphi,\psi\rangle =  \overline{ \langle \Upsilon_{j,k} \psi, \varphi\rangle}.
\]
 
Let   $A=\Upsilon_{0,0},\, C=\Upsilon_{N,N}$, and 
\[
\begin{split}
 B & =\begin{pmatrix} \Upsilon_{j,k} \end{pmatrix}_{j,k=1}^{j,k = N-1}  
 \\ \pE & = \begin{pmatrix} \Upsilon_{0,1} & \Upsilon_{0,2} & \dots & \Upsilon_{0,N-1} \end{pmatrix}
  \\ \pF^* & = \begin{pmatrix} \Upsilon_{1,N}^* & \dots &\Upsilon_{N-2,N}^* & \Upsilon_{N-1,N}^{*} \end{pmatrix}.
\end{split}
\]
  The  matrix $B$ defines a form $[\cdot,\cdot]_B$ on
  $ \mathscr{F}^{N-1} \times \mathscr{F}^{N-1},$ where
  $ \mathscr{F}^{N-1} := \oplus_{1}^{N-1} \mathscr{F},$ 
  in the natural way. Likewise,  $\pE$ defines a form on $\mathscr{F}^{N-1} \times \mathscr{F}$;
  and $\pF$ defines a form on  $\mathscr{F} \times \mathscr{F}^{N-1}$. 
 For instance, given $\varphi^\prime=\oplus_1^{N-1} \varphi^\prime_j \in \mathscr{F}^{N-1 }$ and similarly
 for $\psi^\prime \in\mathscr{F}^{N-1},$
\[
 [\varphi^\prime,\psi^\prime]_B=  \langle B \varphi^\prime, \psi^\prime \rangle 
   = \sum_{j,k=1}^{N-1}  \langle \Upsilon_{j,k} \varphi^\prime_k,\,  \psi^\prime_j \rangle 
     = \sum_{j,k=1}^{N-1}  \sum_{g,h\in G} \blangle p_{j,k}(g^{-1}h) \varphi^\prime_k(h), \,  \psi_j^\prime(g) \brangle 
\]
(where the sums are finite)
 and $B$ is \df{positive semi-definite (psd)}  \index{psd}
 means $\langle B\varphi^\prime,\varphi^\prime\rangle  \ge 0$ for all $\varphi^\prime \in \oplus_1^{N-1} \mathscr{F}.$

The following is an analog of the psd version of the Parrott theorem. 
\begin{theorem}
\label{p:PsD-Parrot-for-forms}\index{Parrott theorem, psd form}
    If both the forms on $\mathscr{F}^N \times \mathscr{F}^N$ 
\[
 P=  \begin{pmatrix} A & \pE \\ \pE^* & B \end{pmatrix}, \ \ \ 
 Q=  \begin{pmatrix} B & \pF \\ \pF^* & C \end{pmatrix}
\]
 are psd (and if, as above,   $A$, $B$, and $C$ are positive definite (pd)),  
 then there is a function $p:G\to B(\HE)$ such that,
 with $\Upsilon_{0,N} = \Upsilon_{p}$ and $\Upsilon_{N,0} = \Upsilon_p^* = \Upsilon_{p^*}$, the form $\Upsilon=(\Upsilon_{j,k})_{j,k=0}^{N}$
 on $\mathscr{F}^{N+1},$
\[
 \Upsilon = \begin{pmatrix} A & \pE & \Upsilon_p
   \\ \pE^* & B & \pF \\ \Upsilon_p^* & \pF^* & C \end{pmatrix},
\]
  is psd.
\end{theorem}

\begin{proof}
  The pd form $B$ defines a positive definite sesquilinear form 
  $[\cdot,\cdot]_B$ on $\mathscr{F}^{N-1}.$   
  Indeed, that is precisely what it means to say $B$ is pd. 
  Let $\cM_B$ denote the inner product space $(\mathscr{F}^{N-1},[\cdot,\cdot]_B)$ and 
  let $\HE_B$ denote the Hilbert space obtained by completing $\cM_B.$ 
  In the same manner, let $\cM_A$ and $\cM_C$ denote the inner product spaces
  on $\mathscr{F}$ induced by the pd forms $[\cdot,\cdot]_A$ and $[\cdot,\cdot]_C$
  and let $\HE_A$ and $\HE_C$ denote the Hilbert spaces obtained by completing $\cM_A$ and $\cM_C$ 
  respectively. By construction, $\mathscr{F} \subseteq \cM_A, \cM_C$
and $\mathscr{F}^{N-1} \subseteq \cM_{B}$.
   
   The linear map $I_{N-1}\otimes L_\lralpha,$ where $L_\alpha$ is defined
   in equation~\eqref{e:Lalpha}, acts   on $\mathscr{F}^{N-1}$ in the natural way;  that is, for 
  $\varphi^\prime =\oplus_{j=1}^{N-1} \varphi^\prime_j,$
\[
  (I_{N-1}\otimes L_{\lralpha}) \varphi^\prime =\oplus_{j=1}^{N-1} L_{\lralpha} \varphi^\prime_j.
\]
  Often we simplify and write $L_\alpha$ in place of $I_{N-1}\otimes L_\alpha.$ 
   From equation~\eqref{e:YpLx}, it follows,  
   for $\varphi,\psi\in \mathscr{F}$ and $\varphi^\prime,\psi^\prime  \in \mathscr{F}^{N-1},$ that
\begin{align}
 [L_\lralpha \varphi, L_\lralpha \psi]_A =  \langle A L_\lralpha \varphi, L_\lralpha \psi \rangle & =
 { \langle \Upsilon_{0,0} L_\lralpha \varphi, L_\lralpha \psi \rangle=
 \langle \Upsilon_{0,0} \varphi, \psi\rangle }=    \langle A \varphi, \psi\rangle  = [\varphi, \psi]_A  
  \label{e:invariant:form0} 
\\  \langle \pE L_\lralpha \varphi^\prime, L_\lralpha \psi \rangle  & = 
{  \sum_{j=1}^{N-1} \langle \Upsilon_{0,j} L_\lralpha    \varphi'_j, L_\lralpha \psi     \rangle =
\sum_{j=1}^{N-1} \langle \Upsilon_{0,j}     \varphi'_j,  \psi     \rangle
}= 
\langle \pE \varphi^\prime, \psi \rangle  \label{e:invariant:forms1} 
\end{align}
and 
\begin{equation}
    \label{e:invariant:forms2}
\begin{split}
 [L_\alpha \varphi^\prime, L_\lralpha \psi^\prime]_B
 & = \langle B L_\lralpha \varphi^\prime, L_\lralpha \psi^\prime \rangle   = 
\sum_{j,k=1}^{N-1} \langle \Upsilon_{j,k} L_\lralpha    \varphi'_k, L_\lralpha \psi'_j  \rangle 
\\[3pt] & =
\sum_{j,k=1}^{N-1} \langle \Upsilon_{j,k}  \varphi'_k,  \psi'_j     \rangle
  =  \langle B \varphi^\prime,  \psi^\prime \rangle = [\varphi^\prime, \psi^\prime]_B. 
\end{split}
\end{equation}
 In particular, $L_\lralpha$ induces unitary operators  $S^A_\lralpha$ and $S^B_\lralpha$ on $\cM_A$ and $\cM_B$ respectively 
 that then extend to unitary operators on $\HE_A$ and $\HE_B,$ still denoted $S^A_\lralpha$ and $S^B_\lralpha.$
 For instance, for $\varphi,\psi\in \mathscr{F}\subseteq \HE_A,$ an application of equation~\eqref{e:invariant:form0}
 gives,
\[
 [S_\lralpha^A \varphi, S_\lralpha^A \psi]_A=
 [L_\lralpha \varphi, L_\lralpha \psi]_A = [\varphi, \psi]_A. 
\]

 Because $P$ is psd, for $\psi \in \mathscr{F}$ and $\varphi^\prime \in\mathscr{F}^{N-1},$
  a version of the Cauchy-Schwarz inequality gives, 
\begin{equation}
\label{cauchySchwartz}
 \big |\langle  \pE \varphi^\prime, \psi \rangle \big |^2 \le [\psi,\psi]_A \, [\varphi^\prime,\varphi^\prime]_B.
\end{equation}
 To prove this claim note that, for complex numbers $\lambda$ 
\[
\begin{split}
   \langle P \begin{pmatrix} \psi\\ - \lambda \varphi^\prime \end{pmatrix}, \,
      \begin{pmatrix} \psi \\ - \lambda \varphi^\prime \end{pmatrix} \rangle
    & =  \langle A\psi, \psi \rangle - 2\real ( \lambda \langle \pE \varphi^\prime,\psi \rangle) + |\lambda|^2 \langle B\varphi^\prime,\varphi^\prime\rangle
    \\ & = [\psi,\psi]_A - 2\real (\lambda \langle \pE \varphi^\prime,\psi \rangle) + |\lambda|^2 [\varphi^\prime,\varphi^\prime]_B.
\end{split}
\]
 If $[\varphi^\prime,\varphi^\prime]_B=0,$ then ($\varphi^\prime=0$ and) the positivity hypothesis on $P$ 
 implies it is also
 the case that $\langle \pE\varphi^\prime,\psi\rangle =0$ so that \eqref{cauchySchwartz} holds. Otherwise,  
 the positivity condition and the usual judicious choice 
 $ [\varphi^\prime,\varphi^\prime]_B \, \lambda= \overline{\langle \pE\varphi^\prime,\psi \rangle}$  does the job.
 Hence, for $\varphi^\prime$ fixed, the mapping $\HE_A \ni \psi \mapsto \overline{\langle \pE \varphi^\prime, \, \psi \rangle}$ 
 defines a bounded linear functional on $\cM_A,$ and hence on  $\HE_A,$ of norm 
 at most $\|\varphi^\prime\|_B = [\varphi^\prime,\varphi^\prime]_B^{\frac12}.$
 By the Riesz Representation theorem, there is a vector $\widetilde{\pE} \varphi^\prime \in \HE_A$
 such that 
\begin{equation}
    \label{e:parrott:CS}
 \langle  \pE \varphi^\prime, \psi \rangle = [\widetilde{\pE} \varphi^\prime, \psi ]_A
\end{equation}
 for all $\psi \in \HE_A$ and $\|\widetilde{\pE} \varphi^\prime \|_A \le \|\varphi^\prime\|_B.$  
 Thus, we obtain a bounded linear map  $\widetilde{\pE}:\HE_B\to \HE_A$ 
 with norm at most one satisfying
\begin{equation}
    \label{d:tpE}
 [\widetilde{\pE} \varphi^\prime, \psi]_A = \langle \pE \varphi^\prime, \psi\rangle
\end{equation}
 for all $\psi \in \mathscr{F}$ and $\varphi^\prime\in\mathscr{F}^{N-1}.$ 
 A similar construction produces a (linear) contraction operator $\widetilde{\pF}:\HE_C\to \HE_B$
 such that
\begin{equation}
    \label{d:tpF}
 \langle \pF\psi, \varphi^\prime \rangle = [\widetilde{\pF}\psi, \varphi^\prime]_B.
\end{equation}

 Following  
  the argument in \cite{smith}, set  $\widetilde{X}=\widetilde{\pE}\widetilde{\pF}$
  and note, 
  since $\widetilde{\pF}$ and $\widetilde{\pE}$ are contractions,
\begin{equation}
    \label{e:parrott:R}
\widetilde{R}:= \begin{pmatrix} I & \widetilde{\pE} & \widetilde{X} \\ \widetilde{\pE}^* & I & \widetilde{\pF}\\
    \widetilde{X}^* & \widetilde{\pF}^* & I \end{pmatrix}
    = \begin{pmatrix}  \widetilde{\pE}\\ I \\ \widetilde{\pF}^* \end{pmatrix}
    \,  \begin{pmatrix}  \widetilde{\pE}^*  & I & \widetilde{\pF} \end{pmatrix}
     + \begin{pmatrix} I-\widetilde{\pE} \widetilde{\pE}^*  &0&0\\0&0&0\\0&0& I-\widetilde{\pF}^* \widetilde{\pF} \end{pmatrix}
      \succeq 0.
\end{equation} 

 Next observe, for  $\psi \in \mathscr{F}$ and $\varphi^\prime  \in \oplus_1^{N-1} \mathscr{F},$ that 
 for $\lralpha\in G,$ 
\[
 \begin{split}
[\widetilde{\pE} S^B_{\lralpha} \varphi^\prime, S^A_\lralpha \psi]_A
 & =  \langle \pE L_\lralpha  \varphi^\prime, L_\lralpha  \psi \rangle \\
  & = \langle \pE \varphi^\prime, \psi \rangle 
 =  [\widetilde{\pE} \varphi^\prime, \psi ]_A = [S_{\lralpha}^A \widetilde{E} \varphi', S_{\lralpha}^A \psi]_A,
  \end{split}
\]
 where  the first and third equalities result from the definition 
 of $\widetilde{\pE}$ in equation~\eqref{d:tpE}
  and the definitions of $S^A_{\lralpha}$ and $S^B_\lralpha$; 
  and the second equality follows from 
  equation~\eqref{e:invariant:forms1}.
 The surjectivity of $S^A_\alpha$ now gives
 $ \widetilde{\pE} S^B_\lralpha  =  S^A_\lralpha \widetilde{\pE}.$  A similar argument
 reveals that $\widetilde{\pF} S^C_\lralpha = S^B_\lralpha \widetilde{\pF}.$  Consequently,
$\widetilde{\pE}\widetilde{\pF} S^C_\lralpha = \widetilde{\pE} S_\lralpha^B \widetilde{\pF} = 
S^A_\lralpha \widetilde{\pE}\widetilde{\pF};$
  that is $ \widetilde{X} S^C_\lralpha= S^A_\lralpha \widetilde{X}.$ Thus, $(S_\lralpha^A)^* \widetilde{X} S_\lralpha^C = \widetilde{X}$.
  
  Define $p=p_{0,N}: G\to B(\HE)$ by 
\[
 \langle p(h) \lcal f,\lcal e\rangle = [ \widetilde{X} (\mbo_h\otimes \lcal f), (\mbo_e\otimes \lcal e) ]_A,
\]
for $\lcal f,\lcal e\in \HE$
  and $h \in G,$ 
where
$e$ is the group identity and  $\mbo_e$ and $\mbo_h$ are the indicator functions of $\{e\}$ and $\{h\}$ respectively.
Let 
\[  
   X= \Upsilon_{0,N} = \Upsilon_{p}  = \begin{pmatrix} p(g^{-1} h)\end{pmatrix}_{g,h\in G}. 
\]
  Thus, $X$ is a  matrix indexed by $G\times G$ with $(h,g)$ entry  $p(g^{-1}h) \in B(\HE)$ that determines a form
 on $\mathscr{F}.$  We claim if  $\varphi, \psi \in \mathscr{F}$, then 
 $\langle X \varphi, \psi \rangle = [\widetilde{X}\varphi, \psi]_A.$ 
 To prove this claim, let $g,h\in G$ and $\lcal e,\lcal f\in\HE$ be given and 
 observe,
\[
 \begin{split}
 \langle  X  (\mbo_{h} \otimes \lcal f),(\mbo_g \otimes \lcal e)  \rangle  & =
    \langle X_{g,h} \lcal f, \lcal e \rangle 
     = \langle p(g^{-1} h)\lcal f, \lcal e \rangle  
 \\[5pt] &    = [\widetilde{X} (\mbo_{g^{-1}h} \otimes \lcal f), (\mbo_e \otimes \lcal e)]_A 
 \\[5pt]
  & = [ (S^A_{g})^* \widetilde{X} S^C_{g} (\mbo_{g^{-1}h} \otimes \lcal f), (\mbo_e \otimes \lcal e) ]_A
\\[5pt] &  = [ \widetilde{X} S^C_{g} (\mbo_{g^{-1}h} \otimes \lcal f),S^A_{g} (\mbo_e \otimes \lcal e) ]_A
\\[5pt] &  = [ \widetilde{X}  (\mbo_{h} \otimes \lcal f),(\mbo_g \otimes \lcal e) ]_A. 
\end{split}
\]
Because the set $\{\mbo_g \otimes \lcal f: g \in G, \, \lcal f\in \HE \}$ 
spans $\mathscr{F}$, the claim follows by linearity.

To prove  $\Upsilon=(\Upsilon_{j,k})_{j,k=0}^{j,k=N}$ is psd, 
observe,  if $\varphi_1,\varphi_2, \psi_1,  \psi_2 \in \mathscr{F}$ 
 and $\varphi^\prime, \psi^\prime \in \mathscr{F}^{N-1}$, 
 then 
\begin{equation}
\label{e:parrott:1}
 \langle \Upsilon[\varphi_1 \; \;  \varphi^\prime \; \; \varphi_2]^T,  
   [\psi_1 \; \;  \psi^\prime \; \; \psi_2]^T \rangle  = 
   \langle \widetilde{R}[\varphi_1 \; \;  \varphi^\prime \; \; \varphi_2]^T,  [\psi_1 \; \;  \psi^\prime  \; \; \psi_2]^T \rangle.
\end{equation}
 Since, by equation~\eqref{e:parrott:R}, the operator $\widetilde{R}$ is psd, equation~\eqref{e:parrott:1} 
 implies the form $\Upsilon$ is too. 
\end{proof}

\begin{remark}\rm 
   The pd assumption on $A,B,C$ can be relaxed to psd via the usual device of observing that $[\cdot,\cdot]_A$ determines a psd form on $\mathscr{F}$ and constructing
   the Hilbert space $\HE_A$ from equivalence classes, modulo null vectors,  with representatives
   from $\mathscr{F}.$ 
\end{remark}

\begin{corollary}
\label{c:PsD-Parrott}
  Fix positive integers $M,N$ and let $E= \{(j,N),(N,j): -M\le j \le 0\}.$
  Suppose $p_{j,k}:G\to B(\HE)$ for  $(j,k)\in \{(j,k): -M\le j,k\le N\}\setminus E.$ 
  Let $\Upsilon_{j,k}$ denote the form determined by $p_{j,k}.$ If both 
  forms 
\[
 \begin{pmatrix} \Upsilon_{j,k} \end{pmatrix}_{j,k=-M}^{N-1},  \ \ \ 
 \begin{pmatrix} \Upsilon_{j,k} \end{pmatrix}_{j,k=1}^N
\]
 are psd, then there exists $p_{j,k}:G\to B(\HE)$ for $(j,k)\in E$ such that, defining
 $\Upsilon_{j,k}$ for $(j,k)\in E$ in the usual way, the form
\[
\begin{pmatrix} \Upsilon_{j,k} \end{pmatrix}_{j,k=-M}^N 
\]
 is psd.
\end{corollary}

\begin{proof}
 Applying Theorem~\ref{p:PsD-Parrot-for-forms},  to the data $\Upsilon_{j,k}$ 
 for $(j,k)\in \{(j,k): 0\le j,k\le N\}\setminus\{(0,N),(N,0)\},$ it follows 
 that there exists a form $\Upsilon_{0,N}$ on $G$ such that, with $\Upsilon_{N,0}=\Upsilon_{0,N}^*,$  the form
 $(\Upsilon_{j,k})_{j,k=0}^N$ on $\oplus_{0}^N \mathscr{F}$  is psd. 
 Now induct.
\end{proof}

Corollary~\ref{c:PsD-Parrott} has a convenient  restatement purely in terms of functions from $G$ to  $B(\HE)$ as follows.
Given a finite set $F$ let $M_F(B(\HE))$ denote the matrices indexed by $F\times F$
with entries from $B(\HE).$ Given  functions $p_{s,t}:G\to B(\HE)$ for $s,t\in F,$
the map $p:G\to M_F(B(\HE))$ defined by
\[
   p(g) =\begin{pmatrix} p_{s,t}(g) \end{pmatrix}_{s,t\in F}
\]
 is \df{psd} if, for all  
 $\varphi:F\to \mathscr{F},$ 
\[
 \sum_{g,h\in G} \langle p(g^{-1}h) \varphi_h,\varphi_g\rangle 
 := \sum_{g,h\in G} \, \sum_{s,t\in F} \langle p_{s,t}(g^{-1} h) \varphi_h(t),\varphi_g(s) \rangle \ge 0,
\]
where $\varphi_h: F \to \HE$ is defined by $\varphi_h(s) = \varphi(s)(h)$ for each $s \in F$.
 The \df{adjoint} of $p:G\to M_F(B(\HE)),$ denoted $p^*,$ is the 
 function $p^*:G\to M_F(B(\HE))$ defined by
 $p^*(g) =p(g^{-1})^*.$  In particular, given %
 $\varphi,\psi:F\to \mathscr{F},$ 
\[
 \langle \varphi_h, p^*(h^{-1}g) \psi_g \rangle 
  = \langle \varphi_h, p(g^{-1}h)^* \psi_g \rangle 
  = \langle p(g^{-1}h) \varphi_h,\psi_g \rangle,
\]
for $g,h\in G.$

\begin{corollary}
\label{c:psd-Parrott2}
 Let $F$ denote a finite set,  fix $s_0\in F$ and a proper subset $F_0\subseteq F {\setminus \{ s_0 \}}$
 and suppose $p_{s,t}:G\to B(\HE)$ for 
\[
  (s,t)\in F\times F \setminus \big (  (\{s_0\}\times F_0) \, \cup \,  (F_0\times \{s_0\}) \big )
\]
are given. If both 
\[
 \begin{pmatrix} p_{s,t} \end{pmatrix}_{s,t\in F\setminus \{s_0\}}, \ \ \ 
 \begin{pmatrix} p_{s,t} \end{pmatrix}_{s,t\in F\setminus F_0}
\]
 are psd, then there exists $p_{s_0,t}:G\to B(\HE)$ {for $t \in F_0$} such that, with $p_{t,s_0}= p_{s_0,t}^*,$ 
\[
 \begin{pmatrix} p_{s,t} \end{pmatrix}_{s,t\in F}
\]
 is psd.
\end{corollary}

In applying Theorem~\ref{p:PsD-Parrot-for-forms} and its corollaries, one is not presented
with the matrices $\Upsilon$ directly, but instead  matrices $\Gamma$ indexed by a monoid $M\subseteq G$
with the property $G=\Her M.$ Thus $\Gamma$ is indexed by $M\times M$ and the $(g,h)$ entry of
$\Gamma$ depends only upon $g^{-1}h.$  In this case, there exists a function $p:\Her M\to  B(\HE)$
such that 
\begin{equation}
    \label{e:fromMtoG}
 \Gamma = \Gamma_p := \begin{pmatrix} p(g^{-1}h)\end{pmatrix}_{g,h\in M}.
\end{equation}

\begin{corollary}
\label{c:Gpsd-parrott}
Let $M$ be a monoid and assume that $G = \Her M = \{u^{-1}v : u,v\in M\}$. 
With the hypotheses of the preamble of Corollary~\ref{c:PsD-Parrott}, 
if both forms
\[
 A=  \begin{pmatrix} \Gamma_{j,k} \end{pmatrix}_{j,k=-M}^{N-1},  \ \ \ 
 B=  \begin{pmatrix} \Gamma_{j,k} \end{pmatrix}_{j,k=1}^N
\]
 are psd, then there exists $p_{j,k}:G\to B(\HE)$ for $(j,k)\in E$ such that, defining
 $\Gamma_{j,k}$ for $(j,k)\in E$ in the usual way, the form
\[
\begin{pmatrix} \Gamma_{j,k} \end{pmatrix}_{j,k=-M}^N 
\]
 is psd.
\end{corollary}

The proof of Corollary~\ref{c:Gpsd-parrott} employs Lemma~\ref{l:MonoidG+},
which in turn relies on Lemma~\ref{l:MonoidG} below. Lemma~\ref{l:MonoidG}, 
while likely known as it is very much in line with left Ore quotient construction in 
non-commutative ring theory,  is natural from the point of view of semigroups.

\begin{lemma}
\label{l:MonoidG}
 If ${M}\subseteq G$ is a monoid and $G=\Her {M},$ then for each finite subset $S$ of $G,$ there exists
 a $d\in {M}$ such that $dS=\{ds: s\in S\}\subseteq {M}.$
\end{lemma}

\begin{proof}
 We argue by induction.
 The case where the cardinality of $S$ is one is evident.  Now suppose $m$ is a positive 
 integer and the result holds for all $S$ with $|S|=m.$ Fix an $S=\{s_1,\dots,s_m\}$
 and let $s=u^{-1}v\in G=\Her {M}$ be given. Set $S^\prime =S\cup\{s\}.$
 By assumption, there exists a $d\in {M}$ such that $T=dS =\{t_1,\dots,t_m\}\subseteq {M}.$
 Since $G=\Her {M}$ there exist $a,b\in {M}$ such that $d u^{-1} = a^{-1} b.$  Set $e = ad$ and 
 observe,
\[
 e s = ad s = ad u^{-1}v = b v\in {M}.
\]
Hence $eS^\prime =eS \cup\{es\} \subseteq {M}$ and the proof is complete.
\end{proof}

The statement of Lemma~\ref{l:MonoidG+} below uses the notation of equation~\eqref{e:fromMtoG}.

\begin{lemma}
\label{l:MonoidG+}
   Let $M$ be a monoid and assume that $G = \Her M$.
   If $N$ is a positive integer,  $p_{j,k}: G\to  B(\HE)$ for $1\le j,k\le N$ 
   and 
\[
   \Delta = \begin{pmatrix} \Gamma_{p_{j,k}} \end{pmatrix}_{j,k=1}^N
\]
 is psd, then so is
\[
 \Upsilon = \begin{pmatrix} \Upsilon_{p_{j,k}} \end{pmatrix}_{j,k=1}^N,
\]
where
\[
 \Upsilon_p = \begin{pmatrix} p(g^{-1}h) \end{pmatrix}_{g,h\in G}.
\]
\end{lemma}

\begin{proof}
  Let $\varphi = \oplus_1^{N} \varphi_j \in \mathscr{F}^N$ be given,
  where $\mathscr{F}$ is defined in equation~\eqref{eq:funnyF}.
 Let $S = \bigcup_{1 \leq j \leq N} S_j$ where $S_j = \text{support}(\varphi_j)$.
 By Lemma~\ref{l:MonoidG}, there exists a $d\in {M}$ such that $T=dS\subseteq {M}.$ 
 Define $\psi_j:M\to \HE$ by $\psi_j(g)= \varphi_j(d^{-1}g).$  Thus $\psi_j$ is supported in $T$ and
 \begin{equation*}
\begin{split}
\langle \Upsilon \varphi,\varphi\rangle 
&= \sum_{j,k=1}^{N}  \sum_{g,h\in S} \blangle p_{j,k}(g^{-1}h) \varphi_k(h), \,  \varphi_j(g) \brangle
\\[3pt]  & = \sum_{j,k=1}^{N}  \sum_{g,h\in S} \blangle p_{j,k}((dg)^{-1} (dh)) \psi_k(dh), \,  \psi_j(dg) \brangle
\\[3pt]  & = 
\sum_{j,k=1}^{N}  \sum_{r,s\in  T} \blangle p_{j,k}(r^{-1} s)\psi_k(s), \,  \psi_j(r) \brangle
\\[3pt] & = \langle \Delta \psi,\psi\rangle \succeq 0,
\end{split}
\end{equation*}
where $\psi =\oplus \psi_j.$ 
\end{proof}

\begin{proof}[Proof of Corollary~\ref{c:Gpsd-parrott}]
 Apply Lemma~\ref{l:MonoidG+} to both $A$ and $B$ to conclude
 their $\Upsilon$ counterparts in Corollary~\ref{c:PsD-Parrott} 
 are psd. Hence by that corollary $(\Upsilon_{j,k})_{j,k=-M}^N$ is psd
 and therefore so is $(\Gamma_{j,k})_{j,k=-M}^N$ as claimed.
\end{proof}

 We conclude this section with the following variation on Corollary~\ref{c:psd-Parrott2}. 

\begin{corollary} 
\label{psd-parrotMonoid}
 Suppose  $M$ is a monoid and $G = \Her M = \{u^{-1}v : u,v\in M\}$. 
 Let $F$ denote a finite set,  fix $s_0\in F$ and a proper subset $F_0\subseteq F {\setminus \{ s_0 \}}$
 and suppose $p_{s,t}:G\to B(\HE)$ for 
\[
  (s,t)\in F\times F \setminus \big (  (\{s_0\}\times F_0) \, \cup \,  (F_0\times \{s_0\}) \big )
\]
are given. If both 
\[
 \begin{pmatrix} p_{s,t} \end{pmatrix}_{s,t\in F\setminus \{s_0\}}, \ \ \ 
 \begin{pmatrix} p_{s,t} \end{pmatrix}_{s,t\in F\setminus F_0}
\]
 are psd, then there exists $p_{s_0,t}:G\to B(\HE)$ {for $t \in F_0$} such that, with $p_{t,s_0}= p_{s_0,t}^*,$ 
\[
 \begin{pmatrix} p_{s,t} \end{pmatrix}_{s,t\in F}
\]
 is psd.
\end{corollary}

\newcommand{\HerIw}{\Her \groupWw}
\newcommand{\HerIs}{\Her \groupW_{\le \Ws}}
\newcommand{\FK}{F_K}
\newcommand{\FE}{F_{\HE}}
\newcommand{\groupYG}{G}

\section{Matrix completions}
\label{s:freezedtwo}

The positive semidefinite (psd) matrix completion theorems for partially defined forms over $\Her\groupWw$,
building on the psd Parrott theorem, Theorem~\ref{p:PsD-Parrot-for-forms} from Section~\ref{sec:parrott},
are formulated and proved in this section. 
These completion results provide the combinatorial input for the representation theorem in Section~\ref{sec:psd_functions_on_groups_arise_from_unitary_representations} and, ultimately, for the proof of the Fej\'er--Riesz factorization in Section~\ref{sec:proofs}.

Let $\groupYG$ be a group and let \df{$F(\groupYG,B(\HE))$} denote the set of functions $\vartheta:\groupYG\to B(\HE).$
  We abbreviate to \df{$\FE(\groupYG)$} when confusion is unlikely.  Recall that the function $\vartheta$ 
is identified with the matrix, or form, with entries from $B(\HE)$, \index{$ \Upsilon_\vartheta$}
\[
 \Upsilon_\vartheta = \begin{pmatrix} \vartheta( g^{-1} h)  \end{pmatrix}_{g,h\in \groupYG}.
\]

Also recall $\groupW$ is either the free semigroup $\llx_\vg$ or the free product $\Ztg$
endowed  with the shortlex order.
Given $\Ww\in \groupW,$ 
a \df{partially defined form}, or \df{$w$-partially defined form}, 
over $\HerIw$ with values  in $\FE(\groupYG)$ is a function
$\chi: \HerIw \to \FE(\groupYG).$ The function $\chi$ is identified with the block  matrix $X_\chi$\index{$X_\chi$}
indexed by $\groupWw \times \groupWw,$ 
\[
 X_\chi = \begin{pmatrix}  \Upsilon_{\chi(\Wu^{-1}\Wv)} \end{pmatrix}_{\Wu,\Wv\in \groupWw} 
 \sim \begin{pmatrix}  \chi(\Wu^{-1}\Wv) \end{pmatrix}_{\Wu,\Wv\in \groupWw},
\]
where $\sim$ denotes the natural entrywise identification.
For notational convenience, let
\[
\chi(g^{-1}h;\Wu^{-1}\Wv) = \chi(\Wu^{-1}\Wv)(g^{-1}h),
\]
 for $g,h\in G$ and $\Wu,\Wv\in \groupWw.$
Let \df{$C_{00}(\groupYG,\HE)$} denote the space of functions $f:G\to \HE$ of finite support. 
With this notation, $\chi$ is  \df{positive semidefinite (psd)} \index{psd} if the matrix $X_\chi$ is
psd; that is, for each
  collection $\{f_\Wv \in C_{00}(\groupYG,\HE) \mid \Wv\in \groupWw\},$ 
\[
 \sum_{\Wu,\Wv\in \groupWw} \sum_{g,h\in \groupYG} \langle  \chi(g^{-1}h;\Wu^{-1}\Wv)  f_\Wv(h), f_\Wu(g) \rangle \ge 0.
\]
 Likewise, a function $\widetilde{\chi}:\Her \groupW \to \FE(\groupYG)$ is psd if, for each $\Wz\in\groupW$ and
 each collection $\{f_\Wv \in C_{00}(\groupYG,\HE) \mid \Wv\in \groupW_{\le \Wz}\},$
\begin{equation}
    \label{e:HerW}
 \sum_{\Wu,\Wv\in \groupW_{\le\Wz}} \sum_{g,h\in \groupYG} \langle  \widetilde{\chi}(g^{-1}h;\Wu^{-1}\Wv)  f_\Wv(h), f_\Wu(g) \rangle \ge 0.
\end{equation}

 The following theorem solves a completion problem for the case of the free semigroup $\llx_\vg.$

\begin{theorem}
\label{l:compHX}\index{matrix completion, free semigroup case}
    In the case of the free semigroup $\llx_\vg$,  
    with notations as above and  letting $\Ws$ denote the immediate
    successor to $\Ww,$  if $\chi,$ a $\Ww$-partially defined form, is psd, then  
     $\chi$ extends to a psd function
    $\widetilde{\chi}:\HerIs \to \FE(\groupYG).$ %
\end{theorem}

\begin{proof}
By definition and the assumption that $\chi$ is psd,
\[
0 \preceq X_\chi =\begin{pmatrix}  \chi(\Wu^{-1}\Wv) \end{pmatrix}_{\Wu,\Wv\in \groupWw} = \begin{pmatrix}  \chi(\Wu^{-1}\Wv) \end{pmatrix}_{\Wu,\Wv\in \groupW_{\le \Ws} \setminus \{ \Ws \} }.
\]
 
Let $F_0 =  \{  \Wv \in \groupW_{\le \Ws} \mid s_{1}^{-1} \Wv \not \in \groupWw \},$ where 
$s_{1}$ is the first letter of $\Ws$. Thus 
 $F_0$ consists only of the empty word and the elements of $\groupWw$ whose first letter is not  $s_{1}.$ 
 Equivalently,  $\groupWw \setminus F_0$ consists of those words
in $\groupWw$ whose first letter is $s_1.$ Moreover, if $\Wt\in F_0,$ then  $\Ws^{-1}\Wt \ne \Wt^{-1}\Ws,$
because both $\Ws^{-1}\Wt$ and  $\Wt^{-1}\Ws$ are in reduced form (no cancellation) and $\Wt\ne \Ws.$

Let $s_{1}^{-1} (\groupW_{\le \Ws} \setminus F_0) = \{  s_{1}^{-1}\Wu  \mid \Wu \in (\groupW_{\le \Ws} \setminus F_0)  \}$. 
Because $s_1^{-1}(\groupW_{\le \Ws} \setminus F_0) \subseteq \groupWw$,  by assumption,
\[
\begin{split}
0 \preceq \begin{pmatrix}  \chi(\Wu^{-1}\Wv) \end{pmatrix}_{\Wu,\Wv\in s_1^{-1}(\groupW_{\le \Ws} \setminus F_0)} 
& = \begin{pmatrix}  \chi(  (s_1^{-1}\Wu)^{-1}s_1^{-1}\Wv) \end{pmatrix}_{\Wu,\Wv\in \groupW_{\le \Ws}  \setminus F_0}
\\[3pt] & = \begin{pmatrix}  \chi(  \Wu^{-1}\Wv) \end{pmatrix}_{\Wu,\Wv\in \groupW_{\le \Ws}  \setminus F_0}.
\end{split}
\]
Given  $(\Wu,\Wv) \in \groupWw  \times \groupWw $, 
let $p_{(\Wu,\Wv)}: \groupYG \to B(\HE)$ denote $p_{(\Wu,\Wv)} = \chi(\Wu^{-1}\Wv)$.
Since  $\Ws \not \in F_0$,   Corollary \ref{c:psd-Parrott2} with $\Ws_0=\Ws$ implies
there exists $p_{\Ws_0,\Wt}: \groupYG \to B(\HE)$ for $\Wt \in F_0$ such that, with $p_{\Wt,\Ws}= p_{\Ws,\Wt}^*,$ 
\[
 \begin{pmatrix} p_{\Wu,\Wv} \end{pmatrix}_{\Wu,\Wv \in \groupW_{\le s} }
\]
 is psd. Now define $\widetilde{\chi}: \Her \groupW_{\le \Ws}  \to \FE(\groupYG)$ by $\widetilde{\chi} \vert_{\Her \groupWw} = \chi$ and if $\Wt \in F_0$, then $\widetilde{\chi}(\Ws^{-1}\Wt) = p_{\Ws,\Wt}$ and 
 $\widetilde{\chi}(\Wt^{-1}\Ws)= p_{\Ws,\Wt}^*$.  
 It follows that 
\[
X_{\widetilde{\chi}} = \begin{pmatrix}  \chi(  \Wu^{-1}\Wv) \end{pmatrix}_{\Wu,\Wv\in \groupW_{\le \Ws} }  =\begin{pmatrix} p_{\Wu,\Wv} \end{pmatrix}_{\Wu,\Wv \in \groupW_{\le \Ws} } \succeq 0.
\qedhere
\]
\end{proof}

 Theorem~\ref{l:completeZ2} below is a companion to Theorem~\ref{l:compHX} in
 that it solves a matrix completion problem for $\Ztg.$
 Since the case $\vg=1$ is trivial, we assume $\vg\ge 2.$
 
\begin{theorem}
\label{l:completeZ2}\index{matrix completion, $\mathbb Z_2^{*{\tt g}}$ case}
    Set $\groupW=\Ztg.$
    With notations as above, if $X_\chi,$ a $\Ww$-partially defined form, is psd, then  letting $\Ws$
    denote the immediate successor of $\Ww$, the function
    $\chi$ extends to a function
    $\overline{\chi} :\HerIs \to \FE(\groupYG)$ such that $X_{\overline{\chi}}$ is also psd. 
\end{theorem}

\begin{proof}
 Since $\Ztg$ is well-ordered, $\Ww$ has an immediate successor $\Ws.$
 Express $\Ws$ in reduced form  as $x_j \Wz$ for a word $\Wz,$  possibly the empty word (the identity).
 In particular, $\Wz < x_j \Wz$ and thus $\Wz\le \Ww.$

  Let $\KL=\groupW_{\le \Wz}\cup x_j \groupW_{\le \Wz}.$ 
  Let  $(\Wa,\Wb) \in \KL \times \KL$ be given. We now determine
  when $\Wa^{-1}\Wb \in \Her \groupWw.$ If $\Wa,\Wb\in x_j\groupW_{\le z},$ then $\Wa^{-1}\Wb \in \Her \groupW_{\le \Wz}$
  and hence $\Wa^{-1}\Wb\in \Her \groupW_{\le \Ww}.$  Now suppose $\Wa \in \groupW_{\le \Wz}$ and  $\Wb \in x_j \groupW_{\le\Wz}.$
  Hence $\Wb=x_j \Wu$ for some $\Wu\le \Wz.$ If $\Wu<\Wz,$ then  $x_j \Wu < x_j \Wz$ and hence $x_j \Wu \le \Ww$
  and is thus in $\groupW_{\le\Ww}.$ Once again $\Wa^{-1}\Wb\in \Her \groupW_{\le \Ww}.$  Consequently, 
  if $\Wa\in \groupW_{\le \Wz}$ and  $\Wb \in x_j \groupW_{\le \Wz},$ but $\Wa^{-1}\Wb \not \in \HerIw$, then $\Wb=x_j \Wz.$ It now follows, from symmetry, if  $(\Wa,\Wb)\in L\times L$ and
  $\Wa^{-1}\Wb \not \in \HerIw,$ then   
  either $\Wa \in \groupW_{\le \Wz}$ and  $\Wb = x_j\Wz;$ or $\Wa = x_j\Wz$ and $\Wb \in \groupW_{\le \Wz}$. Next observe, 
  if $\Wu<\Wz,$ then $\Wu^{-1} x_j \Wz = (x_j\Wu)^{-1} \Wz\in \HerIw,$
  since $x_j \Wu< x_j\Wz$ (since $\Wz$ does not begin with $x_j$ (on the left) the relation is immediate
  if $\Wu$ also does not begin with $x_j$, and if $\Wu$ begins with $x_j,$ then $|x_j \Wu| < |\Wu| \le |\Wz|< |x_j \Wz|$)
  and hence  both $\Wz, \, x_j\Wu$ are in  $\groupWw.$  
  In a similar manner, $(x_j\Wz)^{-1}\Wu= \Wz^{-1} (x_j\Wu) \in \HerIw$ when $\Wu < \Wz$.
  On the other hand, $\Wz^{-1} x_j \Wz\notin \HerIw.$
  Thus if $(\Wa,\Wb)\in L\times L,$ then $\Wa^{-1}\Wb\in \HerIw$ 
  if and only if $(\Wa,\Wb)\notin \{(\Wz,x_j \Wz),(x_j \Wz,\Wz)\}.$

  Let $\KL^\prime = \KL\setminus \{\Ws\}.$  
  For $(\Wa,\Wb)\in \KL\times \KL \setminus \{(\Wz,x_j \Wz),(x_j \Wz,\Wz)\},$
  let $p_{\Wa,\Wb} =\chi(\Wa^{-1}\Wb):\groupYG\to B(\HE).$
  Since $\KL^\prime \subseteq \groupWw,$ by assumption,
\[
 P =\begin{pmatrix} p_{\Wa,\Wb} \end{pmatrix}_{\Wa,\Wb\in \KL'}
\]
 is psd. %
 Since the mapping $\varpi: \KL\setminus \{\Wz\}\to   {\KL'}$ 
 defined by $\varpi(\Wu)=x_j\,\Wu $ is a bijection
 ($L^\prime$ and $L\setminus\{\Wz\}$ have the same cardinality,
 namely $|L|-1,$ and $\varpi$ is easily seen to be onto),
\[
\begin{split}
 Q  & =\begin{pmatrix} p_{\Wa,\Wb} \end{pmatrix}_{\Wa,\Wb\in \KL\setminus \{\Wz\}}
  = \begin{pmatrix} \chi(\Wa^{-1} \Wb) \end{pmatrix}_{\Wa,\Wb\in \KL\setminus \{\Wz\}}
  \\ &  = \begin{pmatrix} \chi((x_j \Wa)^{-1} (x_j \Wb)) \end{pmatrix}_{\Wa,\Wb\in \KL\setminus \{\Wz\}}
  \cong \begin{pmatrix}  \chi(\Wu^{-1}\Wv) \end{pmatrix}_{\Wu,\Wv\in  {\KL'}}
  \\  & = P,
\end{split}
\]
 where $\cong$ means unitarily equivalent (via the spatial unitary conjugation implemented by
 $\varpi$).   Hence $Q$ is also psd.

 Corollary~\ref{c:psd-Parrott2},  applied with $F=L,$ $F_0=\{z\}$ and $s=\Ws,$ 
 now produces a function $p_{\Wz,x_j\Wz}: \groupYG \to B(\HE)$
 such that, setting $p_{x_j\Wz,\Wz}=p_{\Wz,x_j\Wz}^*,$ 
\[
 R= \begin{pmatrix} p_{\Wa,\Wb} \end{pmatrix}_{\Wa,\Wb\in \KL} \succeq 0. 
\]
 It need not be the case that $p_{\Wz,x_j\Wz} = p_{x_j\Wz,z}$ even though $\Wz^{-1} (x_j\Wz)= (x_j\Wz)^{-1}\Wz$ in $\Ztg$ (since $x_j^{-1}=x_j$).
 To remedy this deficiency we argue as follows.
 Observe that the mapping $\KL\ni \Wu \mapsto x_j \Wu\in \KL$ is a bijection. 
 Thus
\[
 R^\prime = \begin{pmatrix} p_{x_j \Wa, x_j \Wb}  \end{pmatrix}_{\Wa,\Wb\in \KL} 
\]
is also psd and thus so is
\[
 \widetilde{R} = \frac12 \left ( R + R^\prime \right ) =
  \frac{1}{2} \begin{pmatrix} p_{a,b} + p_{x_ja,x_jb} \end{pmatrix}_{a,b\in L}.
\]
As was established earlier,  so long as $(x_j\Wa, x_j \Wb) \not\in \{(\Wz,x_j\Wz),(x_j\Wz,\Wz)\},$
equivalently $(\Wa,\Wb)\notin \{(x_j \Wz,\Wz), (\Wz,x_j\Wz)\}$, we have 
$(x_j \Wa)^{-1} (x_j\Wb) = \Wa^{-1}\Wb\in \Her \groupWw.$ Thus  
$p_{x_j \Wa,x_j \Wb} = p_{\Wa,\Wb}=\chi(\Wa^{-1}\Wb)$ and therefore $\widetilde{p}_{\Wa,\Wb} = p_{\Wa,\Wb}.$
On the other hand, 
\[
  2 \widetilde{p}_{x_j\Wz,\Wz} = p_{x_j\Wz,\Wz} + p_{\Wz,x_j\Wz}  
   = p_{\Wz,x_j\Wz} +  p_{x_j\Wz,\Wz} = 2 \widetilde{p}_{\Wz,x_j\Wz}.
\]
Consequently, the function  $\widetilde{\chi}:\HerIw\cup \{\Wz^{-1}x_j \Wz\} \to \FE(\groupYG)$ given by
$\widetilde{\chi}(\Wa^{-1}\Wb) = \widetilde{p}_{\Wa,\Wb}$ is well defined, 
its restriction to $\Her \groupWw$ is $\chi$ and  
both 
\begin{equation}
    \label{e:completeZ2:1}
  \begin{pmatrix} \widetilde{\chi}(\Wa^{-1}\Wb) \end{pmatrix}_{\Wa,\Wb\in \groupWw}, \ \ \ 
  \begin{pmatrix} \widetilde{\chi}(\Wa^{-1}\Wb) \end{pmatrix}_{\Wa,\Wb\in \KL}
\end{equation}
are psd.

We next identify the set $F_0= \groupW_{\le \Ws}\setminus L\subseteq \groupWw.$
To this end, let $n=|\Wz|$ (the length of $\Wz$). Thus $|\Ws|=n+1$ and therefore
if $|\Wu|\le n,$  then $\Wu\le \Ws;$  that is, 
 $\groupW_{\le \Ws}$
contains all words of length at most  $n.$ Moreover, elements of $\groupW_{\le\Ws}$
have length at most $n+1.$ Now suppose $\Wu\in \groupW_{\le \Ws},$
but $\Wu\notin L$.
It is immediate that $\Wz<\Wu \le \Ww <\Ws.$ Thus, if  $\Wu=x_j \Wu_0$ in 
reduced form, 
then   $\Wu_0 <\Wz$ as  $\Wu=x_j \Wu_0  < \Ws = x_j \Wz.$ 
Hence $\Wu \in x_j \groupW_{\le \Wz}\subseteq L.$ From this contradiction, it follows that
there is a $k\ne j$ and a word $\Wu_0$ of length at most $n$ such that 
$\Wu=x_k \Wu_0$ (in reduced form) and $ \Wz < \Wu = x_k \Wu_0 < \Ws.$  

Now suppose $\Wz<\Wu \le \Ww\le \Ws$
 and
there is  a $k\ne j$ and a word $\Wu_0$ of length at most $n$ such that 
$\Wu=x_k \Wu_0.$
In particular, $\Wu\notin \groupW_{\le \Wz}.$
 On the other hand,  if $\Wu \in x_j \groupW_{\le \Wz},$ then  there is a $\Wv\le \Wz$ such that
$x_k \Wu_0 = x_j \Wv.$ Hence $\Wv=x_j x_k \Wv_1$ in reduced form for
some $\Wv_1$ of length at most $n-2.$  We conclude that $x_k \Wu_0 = x_k \Wv_1,$ 
which leads to the contradiction that $|\Wu| \le n-1 < |\Wz|.$  Summarizing,
\begin{equation}
\label{e:FminusL}
 F_0 = \groupW_{\le \Ws}\setminus L = \{ x_k \Wu_0 :  |\Wu_0| \le n, \, k\ne j,  \,  \Wz < x_k \Wu_0 <\Ws\}.\footnote{The set $F_0$
is empty if and only if  $\vg=2$ and $w$ is
the largest word of its length.}
\end{equation}
 Note that $\Wu=x_k\Wu_0$ can be assumed in reduced form, as otherwise, $|\Wu|\le n-1,$ in which case $\Wu < \Wz.$

 Next suppose 
\[
    (\Wu,\Wv)\in \Phi= \groupW_{\le \Ws}\times \groupW_{\le \Ws}\setminus  \left [ (F_0\times \{\Ws\}) \cup (\{\Ws\}\times F_0) \right ].
\]
 If $\Wu\ne \Ws \ne \Wv,$ then $\Wu^{-1}\Wv \in \Her \groupWw.$ Next consider the case that 
 $\Wv=\Ws$ and 
 $\Wu\notin F_0.$ Thus $\Wu\in L.$ If $\Wu<\Wz,$ 
 then $x_j \Wu < x_j\Wz =s.$
 Thus $x_j \Wu \le \Ww$ and hence
\[
  \Wu^{-1}\Wv =  \Wu^{-1} \Ws = \Wu^{-1} x_j \Wz = (x_j \Wu)^{-1} \Wz \in \Her \groupWw.
\]
 On the other hand, if $\Wu\in x_j\groupW_{\le z},$ then it is evident that $\Wu^{-1} \Ws \in \Her \groupWw.$
 Finally, if $\Wu=z,$ then $(\Wu,\Wv)=(\Wz,\Ws)$ and $\Wu^{-1}\Ws = \Wz^{-1} x_j \Wz\in \Her \groupWw \cup\{\Wz^{-1} x_j \Wz\}.$ 
 Hence, in any case, if $\Wv=\Ws$ and $\Wu\notin F_0,$ then $ \Wu^{-1} \Wv \in \Her \groupWw \cup\{\Wz^{-1} x_j \Wz\}.$
 By symmetry, if $\Wu=\Ws$ and $\Wv\notin F_0,$ then $\Ws^{-1}\Wv \in\Her \groupWw \cup\{\Wz^{-1}x_j\Wz\}.$
 Consequently,  $\{\Wu^{-1}\Wv:(\Wu,\Wv)\in\Phi\}\subseteq \Her \groupWw \cup\{\Wz^{-1}x_j\Wz\}$ and 
 $\widetilde{\chi}(\Wu^{-1}\Wv)$
 is defined for $(\Wu,\Wv)\in \Phi.$

 Let  $\Wu\in F_0=\groupW_{\le \Ws}\setminus L$ be given. By equation~\eqref{e:FminusL},   there is a $k\ne j$ and a $\Wu_0$ of length at
 most $n$ and at least $n-1$ such that $\Wu=x_k \Wu_0$ and $\Wz < \Wu = x_k \Wu_0 <\Ws.$  Suppose $\Wa,\Wb \le \Ww$
 and $\Wa^{-1}\Wb = \Ws^{-1} \Wu= \Wz^{-1} x_j x_k \Wu_0.$  
 Since $2n+2 \ge |\Ws^{-1} \Wu| \ge 2n+1$ and
 $|\Wa|, \, |\Wb| \le n+1$, either  $|\Wa|\ge n$ and $|\Wb|=n+1$ 
or $|\Wa|=n+1$ and $|\Wb|\ge n.$ In either case $\Wa^{-1}\Wb = \Ws^{-1} \Wu= \Wz^{-1} x_j x_k \Wu_0$ is in reduced form. 
In the first case either $\Wb=x_k \Wu_0$ and $\Wa = x_j \Wz=\Ws,$ contradicting $\Wa<\Ws;$
or $\Wb= x_j x_k \Wu_0,$ which implies $\Ws = x_j \Wz < x_j x_k\Wu_0 = x_j \Wu =\Wb,$
contradicting $\Wb < \Ws.$  In the other case, $\Wa =\Ws> \Ww,$ contradicting $\Wa \le \Ww.$
At this point we have shown, if $\Wu\in F_0,$ then $\Wu^{-1}\Ws\notin \Her \groupWw$
and by symmetry $\Ws^{-1}\Wu\notin  \Her\groupWw.$  Finally,  $\Ws^{-1}\Wu \neq \Wz^{-1} x_j \Wz=\Ws^{-1}\Wz$
since $\Wz <\Wu.$ Similarly $\Wu^{-1}\Ws\neq \Wz^{-1} x_j \Wz.$  We conclude if $\Wu\in F_0,$
then $\Wu^{-1}\Ws$ and $\Ws^{-1}\Wu$ are not in $\Her \groupWw  \cup \{\Wz^{-1} x_j \Wz\}.$  
Conversely, every element of $\Her \groupWw$ is the image of a pair in
$\groupWw\times\groupWw\subseteq\Phi$, and $\Wz^{-1}x_j\Wz$ is the image
of $(\Wz,\Ws)\in\Phi$. Thus the image of $\Phi$ under
$(\Wu,\Wv)\mapsto\Wu^{-1}\Wv$ is
\[
 \{\Wu^{-1}\Wv:(\Wu,\Wv)\in\Phi\}
 = \Her \groupWw \cup\{\Wz^{-1}x_j\Wz\}.
\]

{At this point the hypotheses of Corollary~\ref{c:psd-Parrott2} for $\groupW_{\le \Ws},\, F_0$ and $s_0=\Ws$ have been validated
{for $\widetilde{p}_{\Wa,\Wb} =\widetilde{\chi}(\Wa^{-1}\Wb)$ for $(\Wa,\Wb)\in \Phi.$  Moreover, 
if $(\Wa,\Wb)\notin\Phi,$ then $\Wa^{-1}\Wb$ is not in the domain of $\widetilde{\chi}.$}
The psd completion promised by the corollary produces entries $p_{\Wu,\Ws}$
and $p_{\Ws,\Wu}$ 
in the $(\Wu,\Ws)$ and $(\Ws,\Wv)$ locations for $\Wu,\Wv\in F_0.$
 Thus, so long as
no two of these entries are required to be the same, we obtain a psd completion of
the partially defined psd function $\widetilde{\chi}.$  
From the definition of $L$ it is immediate that if $\Wu,\Wv\in F_0$ and $\Wu^{-1}\Ws=\Wv^{-1}\Ws,$ 
equivalently, $\Ws^{-1}\Wu= \Ws^{-1} \Wv,$ then $\Wu=\Wv.$  Now suppose $\Wu^{-1}\Ws = \Ws^{-1}\Wv.$
From the description of $F_0$ in equation~\eqref{e:FminusL},  there exists $\Wu_0,\Wv_0$ of length at
most $n$ and $k,\ell\ne j$ such that $\Wu=x_k\Wu_0$ and $\Wv=x_\ell \Wv_0$ (in reduced form) and 
$\Wz< x_k\Wu_0, \, x_\ell \Wv_0 <\Ws.$ Hence
\begin{equation}
    \label{e:completeZ2:2}
  \Wu_0^{-1} x_k x_j \Wz = \Wz^{-1} x_j x_\ell \Wv_0.
\end{equation}
Both words in equation~\eqref{e:completeZ2:2} are 
in reduced form, which,  since $\Wz < x_\ell \Wv_0<\Ws=x_j \Wz$ implies $\Wv_0=\Wz.$ Likewise
$\Wu_0=\Wz$ too {and we reach the contradiction  $x_kx_j=x_jx_\ell.$}
Hence $\widetilde{\chi}$ extends from $\Her \groupWw \cup\{\Wz^{-1}x_j \Wz\}$ to 
a psd function $\overline{\chi}: \Her \groupWx{\Ws}\to F_\HE(G)$ 
as desired.
}
\end{proof}

We point out that Theorem~\ref{l:completeZ2}
fails for 
groups such as $\ZZ_3\ast \ZZ_2$ or $\ZZ_3^{*2}$, see Section \ref{s:examples}.\looseness=-1

\section{Psd functions on groups arise from unitary representations}\label{sec:psd_functions_on_groups_arise_from_unitary_representations}

In this section we produce honest unitary representations from psd kernels:
we show that every psd function 
$p:\Her \groupW \times \groupYG \to  B(\HE)$ is realized as a compression of a unitary representation of $\Her \groupW \times \groupYG$ in Proposition \ref{p:freeY},
 a fact that is well known in the case of a group. 
This representation-theoretic model provides a  bridge from the completion results of Section~\ref{s:freezedtwo} to the factorization theorems proved in Section~\ref{sec:proofs}. 

Let $\groupW$  be as in the introduction. Thus $\groupW$ is either the free semigroup $\llx_\vg$ on $\vg$
letters or the group $\Ztg.$ Let $\groupYG$ denote a group,   $\HE$  a Hilbert space and recall 
a function $p:\Her\groupW \times \groupYG \to B(\HE)$ is \df{positive semidefinite (psd)} if for
each finite subset $\mathscr{F}$ of $\groupW\times \groupYG $  the block operator matrix
\[
  \begin{pmatrix} p(\WGu^{-1}\WGv) \end{pmatrix}_{\WGu,\WGv\in \mathscr{F}} %
\]
 is psd. In Proposition~\ref{p:freeY} below $e$ denotes the identity in the semigroup $\groupW \times \groupYG.$
 
\begin{proposition}
 \label{p:freeY}
    If $p:\Her \groupW \times \groupYG \to  B(\HE)$ is psd, then there is a unitary representation
    $\pi$ of $\Her \groupW \times \groupYG$ on a Hilbert space $\HF$ 
    and {a bounded operator}  $W:\HE\to \HF$ such that
\[
   p(\WGu^{-1}\WGv) = W^* \pi(\WGu^{-1} \WGv) W
\]
for $\WGu,\WGv\in \groupW\times\groupYG.$  Moreover, if $p(e)=I_\HE,$ then $W$ {is an isometry; that is,}
$W^*W=I_\HE.$
\end{proposition}

\begin{remark}\rm 
  Recall %
  that unitary representations $\pi$ 
   of $\Ztg \times \groupYG$ on a Hilbert space $\HF$ correspond to commuting unitary representations 
  $\tau$  of $\Ztg$ and $\rho$ of $\groupYG$ on $\HF.$ In particular, $\tau$ is determined by unitaries
  $U_j=\tau(x_j)$ satisfying $U_j^2=I$ and commuting with $\rho$ in the sense that $U_j\rho(\YGg) = \rho(\YGg)U_j$
  for all $j$ and $\YGg \in \groupYG.$ In this case $\pi(\Wu \YGg)=\tau(\Wu)\rho(\YGg)$ for $\Wu\in \Ztg$ and $\YGg\in \groupYG.$

  In the case of the   free semigroup $\llx_\vg$, there are no constraints on the unitary operators $U_j.$ 
\qed
\end{remark}

\begin{proof}
 In the case $\groupW=\Ztg$, so that $p$ is a psd function on the group $\Ztg\times\groupYG,$
this proposition is a special case of a  standard fact. See for instance \cite[Theorem~4.8]{Paulsen}. 

Suppose now that $\groupW=\llx_\vg$ is the free semigroup (on $\vg$ letters). 
Using the techniques in \cite[Theorem~4.8]{Paulsen} we first show that there exists
 a Hilbert space $\cH,$ an isometry  $\iota: \HE \to \cH$,  a unitary representation
    $\rho$ of $\groupYG$ on $\cH$ and a tuple $(V_1,\dots,V_\vg)$ of isometries   on $\cH$
    that commute with $\rho$ such that 
\[
   p(\Wu^{-1}\Wv\, \YGg^{-1}\YGh) = \iota^* (V^\Wu)^*V^{\Wv}\rho(\YGg^{-1}\YGh) \iota,
\]
for $\Wu,\Wv\in\groupW$ and $\YGg,\YGh\in\groupYG.$

Let $ \mathcal{K} = C_{00}( \groupW \times \groupYG, \HE)$ denote
the semi-inner product space consisting of functions on  $\groupW \times \groupYG$  
{with values in $\HE$} whose support is finite endowed with the positive semidefinite form,
\[
  [f_1, f_2]_\mathcal{K} = \sum_{ \WGu_1, \WGu_2 \in { \groupW \times \groupYG   }} \langle    
   p( \WGu_2^{-1}\WGu_1 ) f_1(\WGu_1), f_2(\WGu_2)          \rangle,
\]
 for $f_1, f_2 \in  C_{00}( \groupW \times \groupYG, \HE)$. 
Define
$ \gamma: \groupW \times \groupYG \xrightarrow{} \mathcal{L}(\mathcal{K}),$
where $\mathcal{L}(\mathcal{K})$ is 
the space of linear maps from $\mathcal{K}$ to itself,  as follows. Given $\lralpha \in \groupW \times\groupYG,$ 
let  $\gamma(\lralpha) = L_{\lralpha^{-1}},$
 where $L_{\lralpha^{-1}}:\mathcal{K}\to \mathcal{K}$ 
 is the linear map defined, for $f\in \mathcal{K},$ by
\[
(L_{\lralpha^{-1}}f)(\WGu)  = 
\begin{cases} 
     f(\lralpha^{-1} \WGu) & \text{if} \; \lralpha^{-1} \WGu \in \groupW \times \groupYG \\
     0 & \text{otherwise}.
  \end{cases}
\]
 {Given $\beta\in \groupW$ observe, because $\groupW=\llx_\vg$ is free,  that $\beta^{-1}\alpha^{-1}\WGu \in \groupW\times G$
 requires $\alpha^{-1}\WGu \in \groupW\times G$ for $\WGu\in \groupW\times G,$ from which it follows
 that $\gamma(\alpha \beta) = \gamma(\alpha)\gamma(\beta).$}

By construction, $\gamma(e)$ is the identity (where $e$ is the identity in $\groupW\times\groupYG$) and $\gamma$ is multiplicative.
Moreover, 
a quick calculation shows, for $\lralpha\in \groupW\times\groupYG,$ %
\begin{equation}
\label{e:descend}
 [f, f]_{\mathcal{K}}    =     [L_{\lralpha}f,L_{\lralpha} f]_{\mathcal{K}}.
\end{equation}
In particular, 
$\gamma(\lralpha)$ is an isometry\footnote{At this point the proof diverges slightly from that in Paulsen where $\gamma(\lralpha)$ is
unitary since the domain of $\gamma$ is a group.}
with respect to the psd form $[\cdot,\cdot]_{\mathcal{K}}$ for all $\lralpha \in \groupW \times \groupYG;$ 
and $\gamma(\YGg)$ is unitary with respect to $[\cdot,\cdot]_{\mathcal{K}}$ for each $\YGg\in\groupYG.$
Moreover, the map $\iota_0:\HE\to\mathcal{K}$ defined, for $\WGu \in \groupW\times\groupYG$ and $\he\in \HE,$  by
\[
\iota_0(\he)[\WGu] =   \begin{cases} 
      \he & \text{if} \; \WGu = e \\
      0 & \text{otherwise,}
   \end{cases}
\]
where $\he\in \HE$ %
is  bounded since 
\[
 [ \iota_0(\he),\iota_0(\he)]_{\cK} = 
 \sum_{\WGu,\WGv \in \groupW\times G} \langle p(\WGu^{-1}\WGv) \iota_0(\he)[{\WGv}], \iota_0(\he)[{\WGu}]\rangle_\HE
  = \langle p(e) \he,\he \rangle \le \|p(e)\| \, \|\he\|^2.
\]

Moreover,  for $\he_1,\he_2\in \HE$ and $\lrbeta_1,\lrbeta_2\in \groupW \times\groupYG,$

\begin{equation}
    \label{e:compress}
\begin{split}
      [ \gamma(\lrbeta_1) \iota_0(\he_1), \gamma(\lrbeta_2)\iota_0(\he_2)]_{\mathcal{K}}
      & = 
    \sum_{\WGu_1,\WGu_2} \langle p(\WGu_2^{-1}\WGu_1) \iota_0(\he_1)[\lrbeta_1^{-1}\WGu_1],\iota_0(\he_2)[\lrbeta_2^{-1} \WGu_2]\rangle 
      \\[3pt] &  = \langle p({\lrbeta_2^{-1}\lrbeta_1})\he_1,\he_2\rangle_\HE.
\end{split}
\end{equation}

Let $\mathcal{N} = \{ f \in C_{00}( \groupW \times \groupYG, \HE) : [f,f]_{\mathcal{K}} = 0 \}$.
The semi-inner product on $\mathcal{K}$ passes to an inner product $\langle \cdot,\cdot\rangle_{\mathcal{K}/\mathcal{N}}$
 on the quotient
${\mathcal{K}/\mathcal{N}}$ as
\[
   \langle [f_1], [f_2] \rangle_{\mathcal{K}/\mathcal{N}} = [f_1, f_2]_{\mathcal{K}},
\]
where  $f_1,f_2\in\mathcal{K}$ and $[f]$ denotes the class of $f$ in $\mathcal{K}/\mathcal{N}.$  %
 Moreover, equation~\eqref{e:descend} 
 implies $\gamma(\lralpha)$ descends to an isometry on $\mathcal{K}/\mathcal{N}.$
  Let $\cH$ be the Hilbert space obtained as the  completion of $\mathcal{K}/\mathcal{N}$  under the inner product $\langle \cdot, \cdot \rangle_{\mathcal{K}/\mathcal{N}}.$ It is immediate that $\gamma(\lralpha)$ extends to an isometry on $\cH,$
  which we continue to denote $\gamma(\lralpha).$  Thus we obtain a multiplicative map, still denoted $\gamma,$ from $\groupW\times \groupYG$ to  $\cB(\cH)$
 such that each $V_j=\gamma(x_j)$ is isometric and 
  $\gamma \vert_{\groupYG}:\groupYG\to \cB(\cH)$  
  is a unitary representation. 
 
{To obtain the desired unitary representation, it suffices
to simultaneously extend each isometry $V_j$ to a unitary  operator
$U_j$ on a larger Hilbert space $\HF$ and at the same time extend $\rho$
to a unitary representation $\widetilde{\rho}$ on $\HF$ such
that $U_j$ and $\widetilde{\rho}(\YGg)$ commute for each $j$ 
and $\YGg\in \groupYG.$

 To construct the $U_j,$ we use the following canonical
 unitary extension of the isometry $V_j$ that represents
 $V_j$ as $U_j$ restricted to an invariant subspace.} 
 Set $\HF=\cH\oplus \cH$ and define 
\[
  U_j = \begin{bmatrix}
   V_j & I - V_j V_j^* \\
   0 & V_j^*  \end{bmatrix}.
\]
 By construction, each $U_j$ is unitary. 
 The mapping $\widetilde{\rho}: \groupYG\to B(\HF)$ defined by
\[
 \widetilde{\rho}(\YGg) = \begin{bmatrix} \rho(\YGg) & 0 \\ 0  & \rho(\YGg) \end{bmatrix}
\]
 for $\YGg\in \groupYG$ is a unitary representation of $\groupYG.$ 
{It remains to prove that $U_j\widetilde{\rho}(\YGg) = \widetilde{\rho}(\YGg)U_j.$ 
 Since $V_j$ commutes with both $\rho(\YGg)$ and $\rho(\YGg^{-1}),$}
\[
V_j \rho(\YGg)^* =  V_j \rho(\YGg^{-1})= \rho(\YGg^{-1}) V_j = \rho(\YGg)^* V_j,
\]
{and therefore} $\rho$ commutes with each $V_j^*.$ {Consequently} $(I-V_jV_j^*)\rho(\YGg)=\rho(\YGg)(I-V_jV_j^{*})$
and thus $\widetilde{\rho}$ commutes with each $U_j.$   
Define  %
 $\pi: \Her \groupW \times \groupYG \xrightarrow{} \cB(\HF)$ by 
\[
   \pi(\Wu^{-1}\Wv, \lrmu^{-1}\lrnu) = U^{\Wu^{-1}\Wv} \widetilde{\rho}(\lrmu^{-1}\lrnu)   
\]
for  %
$\Wu,\Wv\in \groupW$ and $\lrmu, \lrnu \in \groupYG.$
Thus $\pi$ is a unitary
representation of $\Her \groupW \times \groupYG$. 

To complete the proof, let $\iota:\HE\to \cH$ denote the map $\iota(\lcal e)=[\iota_0(\lcal e)],$
the class of $\iota_0(\lcal e)$ in $\mathcal{K}/\mathcal{N}\subseteq \cH;$ define
 $W:\HE \to \HF$ by $W \lcal e= \iota \lcal e \oplus 0;$ and compute, for $\he,\hf \in \HE,$
 using equation~\eqref{e:compress},
\[
\begin{split}
 \langle W^* \pi(\Wu^{-1}\Wv,\lrmu^{-1}\lrnu) W \he,\hf\rangle_\HE
 & =  \langle U^\Wv \widetilde{\rho}(\lrnu) (\iota \he \oplus 0), U^\Wu \widetilde{\rho}(\lrmu) (\iota \hf \oplus 0)\rangle_\HF 
 \\ & = \langle V^\Wv   \rho(\lrnu) \iota \he, V^\Wu \rho(\lrmu) \iota \hf\rangle_{\cH}
 \\ & = [\gamma(\Wv \lrnu) \iota_0\he, \gamma(\Wu \lrmu) \iota_0\hf]_{\mathcal{K}}
 \\ & = \langle p(\Wu^{-1}\Wv,\lrmu^{-1}\lrnu) \he,\hf\rangle_\HE,
\end{split}
\]
and the proof is complete.
\end{proof}

\newcommand{\HH}{\mathbb{H}}

\section{A nested sequence and uniform truncation}
\label{s:nested-sequences}
 The proof of Theorem~\ref{t:main} is accomplished by assigning
 to a given trigonometric polynomial
 $A$ an operator system and completely positive map induced by $A$
 on that operator system. In this section, we identify such an operator system
 via a compactness argument.

Recall that 
$\groupW$ is 
$\llx_\vg$ or $\Ztg$.  
Fixing a $\Ww \in \groupW$ and a positive integer
$M$ for now,  let $\mathscr{S}_\Ww$\index{$\mathscr S$}  denote the set of 
functions $\varphi:\Her\groupWw\to\CC$ identified with the set of matrices
indexed by $\groupWw,$
\[
 X_\varphi = \begin{pmatrix} \varphi(\lru^{-1}\lrv) \end{pmatrix}_{\lru,\lrv\in\groupWw}.
\]
 In a similar manner,  for positive integers $K,$ let $\msTK{M}$\index{$\mathscr T$} denote the set of functions $\psi:\Her\semiyN{M}\to M_K(\CC)$
 identified with the  set of matrices
\[
 \Upsilon_\psi = \begin{pmatrix} \psi(\Ya^{-1}\Yb) \end{pmatrix}_{\Ya,\Yb\in\semiyN{M}}.
\]
 Set $\msT{M}=\msTK{M}$ when $K=1$ and note
 $\msTK{M}$ and  $\msT{M}\otimes M_K(\CC)$ are unitarily equivalent.

For positive integers $W\ge M,$ 
let $\msLK{\Ww}{W} = \mathscr{S}_{\Ww} \otimes \msTK{W}$.\index{$\mathscr L$} Note that 
$\msLK{\Ww}{W}$ {is an operator system that} is  naturally identified (up to unitary equivalence)
with functions $p:\Her \groupWw \times \Her \semiyN{W}\to M_K(\CC).$ 
To such a $p$ we associate a function $\chi_p:\Her \groupWw \to F_K(\Her \semiyN{W}),$
where $F_K(\Her \semiyN{W})$ is
the set of functions from $\Her \semiyN{W}$ to $M_K(\CC),$ defined by
\[
\chi_p(\lru)[\Ya] =\chi_p(\lru,\Ya) = p(\lru \Ya),
\]
for $\Wu\in \Her \groupWw$ and $\Ya \in \Her \semiyN{W}.$ 
As usual, to $p$ we associate the matrix
\begin{equation}
\label{e:Zp}
 Z=Z_p = \begin{pmatrix} p(\WYu^{-1}\WYv) \end{pmatrix}_{\WYu,\WYv\in  \groupWw \times  \semiyN{W}}
   =\begin{pmatrix} \chi_p(\Wu^{-1}\Wv) \end{pmatrix}_{\Wu,\Wv\in \groupWw}.
\end{equation}

Given $Z \in \msLK{\Ww}{W}$, let $Z\vert_{M}$ denote the restriction of $Z$ to $\msLK{\Ww}{M}$.  That is, viewing $Z$ 
as the matrix in equation~\eqref{e:Zp},
\[
 Z\vert_M = \begin{pmatrix} p(\WYu^{-1}\WYv) \end{pmatrix}_{\WYu,\WYv\in \groupWw \times \semiyN{M}}.
\]

Let $\msLK{\Ww}{W}^+$\index{$\mathscr L^+$} denote the psd elements of $\msLK{\Ww}{W}$ viewed equivalently either as psd functions on
$\Her \groupWw \times \Her \semiyN{W}$ or psd matrices $Z$ indexed by { $\groupWw \times \groupWw \times \semiyN{W} \times \semiyN{W}$}
with the property that $Z_{\WYu,\WYv},$ the $(\WYu,\WYv)$ entry of $Z,$ 
depends only upon $\WYu^{-1} \WYv.$ 
Let 
\begin{equation}
    \label{d:fcW}\index{$\mathfrak C_W$}
 \fC_{W} =\{ \mathcal{Z} = Z|_M: Z\in   \msLK{\Ww}{W}^ +, \, Z_{e,e}= I_K\}\subseteq \mathscr{L}^+_{\Ww,M,K}.
\end{equation}

\begin{lemma}
\label{l:nested}
With the notations above, 
 \begin{enumerate}[\rm (i)]
  \item  \label{i:nested:i}
    $\fC_W \supseteq \fC_{W+1};$
  \item \label{i:nested:ii}
    if $Z\in \msLK{\Ww}{W}^+$ and $Z_{e,e}=I,$   then $Z_{g,h} Z_{g,h}^* \preceq I$ for all $g,h\in \groupWw\times \semiyN{W};$
  \item \label{i:nested:iii}
    each $\fC_W$ is compact;
  \item \label{i:nested:iv}
    If $\mathcal{Z}\in \bigcap_W \fC_W,$ then  
   there is a psd function $p:{ \Her \groupW} \times \groupY \to M_K(\mathbb{C}) $ such that 
 \[
   \mathcal{Z}_{\WYu,\WYv}  = p(\WYu^{-1}\WYv) 
\]
  for all $\WYu,\WYv\in \groupWw\times \semiyN{M}.$

\end{enumerate}
\end{lemma}

 \begin{proof}[Proof of item~\ref{i:nested:i}]
 Immediate.
 \end{proof}

 \begin{proof}[Proof of item~\ref{i:nested:ii}]
Let $\fD_W = \{ Z \in \msLK{\Ww}{W}^+ : Z_{e,e} =I_K \} $. %
If $Z \in \fD_W$ and 
$\WYu,\WYv \in \groupWw \times \semiyN{W}$ 
then,
\[
\begin{bmatrix}
I_K & Z_{g,h}^* \\
Z_{g,h} & I_K
\end{bmatrix} %
\]
is a principal submatrix of $Z$ and hence is psd. In particular,
$Z_{g,h}Z_{g,h}^* \preceq I_K.$
\end{proof}

\begin{proof}[Proof of item~\ref{i:nested:iii}]
Continuing with the notation in the proof of item~\ref{i:nested:ii}, put
$n_W=|\groupWw|\,|\semiyN{W}|$ and view $Z$ as an element of
$M_{n_W K}(\CC)$. By item~\ref{i:nested:ii}, each $K\times K$ block
$Z_{g,h}$ has operator norm at most one and hence Frobenius norm at most
$\sqrt K$. Therefore
\[
 \|Z\|_2^2
 = \sum_{g,h\in\groupWw\times\semiyN{W}}\|Z_{g,h}\|_2^2
 \le n_W^2 K.
\]
Thus $\|Z\|_2\le n_W\sqrt K$, so $\fD_W$ is bounded.

Since $\fD_W$ is the intersection of two closed sets --  $\msLK{\Ww}{W}^+$ and 
$\{Z \in  \msLK{\Ww}{W}  : Z_{e,e} = I_K   \}$ -- {it follows that } $\fD_W$ is itself closed 
and hence compact.  Now suppose $(\cZ_n)$ is a sequence from $\fC_W.$ For each $n$
there exists  ${Z_n} \in \fD_W$ such that $\cZ_n={Z_n}\vert_{M}.$
Since $\fD_W$ is compact, the sequence $({Z_n})$ has a subsequential
limit, say ${Z}.$ It follows that $\cZ = {Z}\vert_{M} \in \fC_W$
is a subsequential limit of $(Z_n).$
Thus $\fC_W$ is sequentially compact and hence compact. 
\end{proof}

\begin{proof}[Proof of item~\ref{i:nested:iv}]
If $\mathcal{Z} \in \bigcap_W \fC_W$, then  for each $W \in \mathbb{N}$, there exists $Z^W \in \msLK{\Ww}{W}^+$ such that $Z^W_{e,e} = I_K$ and $Z^W \vert_{M} = \mathcal{Z}$. 
Making use of this collection $\{ Z^W \}$, there exists a sequence $(U^W)$ described as follows: for each $W$,
let $U^W$ denote the matrix indexed by $\groupWw \times \lly$ with $U^W_{\WYu,\WYv}=Z_{\WYu,\WYv}$ if $\WYu^{-1}\WYv \in \Her \groupW \times \Her \semiyN{W}$
and $0$ otherwise.  Thus, while $U^W$ need not be psd, it is the case that $Z^W= U^W\vert_{W}$ is.
In particular, the norm of each entry $U^W_{\WYu,\WYv}$ of $U^W$ has norm at most one by item~\ref{i:nested:ii}.

Since {$\groupY=\Her \lly$ and} $\HerIw \times \groupY \times K^2$ is at most countable, by identifying $U^W$ with a  function whose domain is $\HerIw \times \groupY \times K^2$ and whose codomain is $\mathbb{C},$ (and using the fact that the entries  of $U^W$ are uniformly bounded independent of $W$
and $\WYu^{-1}\WYv$)
there is a subsequence of $(U^W)$ that converges entrywise.   Let $U$ 
denote such a pointwise subsequential  limit. A routine argument shows $U_{e,e}=I_K$ and $U$  
arises from a psd form on {
$\HerIw \times \groupY;$ in other words, there exists a  function $q:\HerIw \times \groupY \to M_K(\CC)$ such that  
$U_{\WYu,\WYv} = q(\WYu^{-1}\WYv)$ for $\WYu,\WYv\in \groupWw \times \lly$ and for  each finite set $F\subseteq \lly,$ the block matrix,
\[
  \begin{pmatrix} q(\Wu^{-1}\Wv, \lrmu^{-1}\lrnu) \end{pmatrix}_{\Wu,\Wv\in \groupWw, \, \lrmu,\lrnu\in F}\succeq 0.
\]
We claim now that $q$ is a psd function in the sense of the 
definition associated with equation~\eqref{e:HerW} preceding Theorem ~\ref{l:compHX}. Indeed,
 if $H\subseteq\groupY$ is a finite set, then, by Lemma~\ref{l:MonoidG},
 there is $d\in \lly$ and a finite set $F\subseteq \lly$ such that
 $F=dH.$ 
 Hence,
\[
\begin{pmatrix} q(\Wu^{-1}\Wv, \lrmu^{-1}\lrnu) \end{pmatrix}_{\Wu,\Wv\in \groupWw, \, \lrmu,\lrnu\in H} 
 = \begin{pmatrix} q(\Wu^{-1}\Wv, \lrmu^{-1}\lrnu) \end{pmatrix}_{\Wu,\Wv\in \groupWw, \, \lrmu,\lrnu\in F}\succeq 0
\]
 and therefore $q$ is a psd function, still denoted $q:\HerIw\times\groupY\to M_K(\CC).$

}

 By construction, $U|_{M}=\mathcal{Z};$ that is,
\[
  q(\WYu^{-1}\WYv) 
  =  Z_{\WYu,\WYv}
\]
 for all $\WYu,\WYv\in \groupWw \times \semiyN{M}.$ %
{ 
 An induction argument based on  
   Theorem~\ref{l:compHX}  in the case of $\groupW=\llx_\vg$ 
  (resp.~Theorem~\ref{l:completeZ2}
 for $\groupW = \Ztg$)
 }
 now shows that 
 $q$  extends to  
 a psd form  
  $p:\Her \llx_\vg\times\groupY \to M_k(\CC)$ (resp.  
  $p: \Ztg \times \groupY  \xrightarrow{} M_K(\CC)$).
 By construction,  $p(\WYu^{-1} \WYv) = q(\WYu^{-1}\WYv)= \mathcal{Z}_{\WYu,\WYv}$ for all $\WYu,\WYv\in \groupWw\times \semiyN{M}$
 and the proof is complete.
\end{proof}

Recall the definition of a $M_K(\CC)$-valued trigonometric polynomial of bidegree at most
$(\Ww,M)$ and its evaluation at a representation from equations~\eqref{eq:trigPoly}
and \eqref{eq:evalPoly}.

Fix a $W\ge M.$ 
Given  $\lralpha\in \Her \groupWw\times \Her \semiyN{W},$ let
$\mbo_\alpha:\Her \groupWw\times \Her \semiyN{W}\to\CC$ denote
the indicator function of $\lralpha.$ For $\WGu,\WGv\in\groupWw\times\semiyN{W}$,
let $E_{\WGu,\WGv}$ denote the matrix with a $1$ in the $(\WGu,\WGv)$ entry
and $0$ elsewhere. Define
\begin{equation}
    \label{d:msBalpha}
\begin{split}
 \msB_\alpha
 & = \begin{pmatrix} \mbo_\alpha(\WGu^{-1}\WGv)\end{pmatrix}_{\WGu,\WGv\in \groupWw\times \semiyN{W}}
 \\ & = \sum_{\substack{\WGu,\WGv\in\groupWw\times\semiyN{W}\\
                       \WGu^{-1}\WGv=\lralpha}} E_{\WGu,\WGv}.
\end{split}
\end{equation}

The collection \df{$\msB$},
\begin{equation}
    \label{d:msB} 
 \msB =\{\msB_\alpha : \lralpha\in \Her \groupWw\times \Her \semiyN{W}\}
\end{equation}
is a basis for $\mathscr{L}_{\Ww,W,1}=\mathscr{S}_{\Ww}\otimes\mathscr{T}_W.$

 Let $\sAwM$\index{$\mathscr A$} denote the set of trigonometric
 polynomials of bidegree at most $(\Ww,M)$ as in equation~\eqref{eq:trigPoly}
 that satisfy the normalization $A_e=I_K,$ where $e$ is the unit in $\Her \groupW\times  \groupY.$
Given $A\in \sAwM,$ and $W\ge M,$  set $A_\alpha=0$ for 
$\lralpha\in \Her \groupWw \times \Her \semiyN{W},$ but  $\lralpha\notin  \Her \groupWw \times \Her \semiyN{M}.$ %
Define $\Phi_A=\Phi_A^W:\mathscr{L}_{\Ww,W,1}\to M_K(\CC)$\index{$\Phi_A$} by
 \begin{equation}
    \label{d:PhisubAW}
      \Phi_A^W(\sB_\alpha) = A_{\alpha}.
\end{equation}
Finally, define $\Phi_A =\Phi_A^\infty: \mathscr{L}_{\Ww,\infty,1}\to M_K(\CC)$
analogously.

\begin{lemma}
\label{t:Kpos}
Suppose $A\in\sAwM$ and  there is an $\epsilon>0$ such that  for each Hilbert space $\HE$
and  unitary representation $\pi:\Her \groupW\times\groupY\to B(\HF),$
\[
 A(\pi)  \succeq \epsilon (I_K \otimes I_{\HF}). 
\]
If $\mathcal{Z} \in \bigcap_W \fC_W$, then $  
(\Phi_A \otimes I_K) (\mathcal{Z}) \succeq \epsilon$. 
\end{lemma}

\begin{proof}
 For $\WYu,\WYv \in \groupWw\times \semiyN{M}$, 
let 
\[
 \mathcal{Z}_{\alpha} = \mathcal{Z}_{\WYu,\WYv},
\]
where $\alpha = \WYu^{-1}\WYv\in \Her \groupWw\times\Her \semiyN{M}.$

By Lemma~\ref{l:nested} item~\ref{i:nested:iv}, there exists
a psd function $p:\Her \groupW \times \groupY \to M_K(\CC)$ such that
\[
 \mathcal{Z}_{\lralpha} = p(\lralpha), \ \ \ \lralpha \in \Her \groupWw \times \Her \semiyN{M}.
\]
 By Proposition~\ref{p:freeY}, 
 there is a Hilbert space $\HF,$ a unitary representation $\pi: \Her \groupW \times\groupY \to B(\HF)$ 
 and an isometry $V:\CC^K\to \HF$ such 
 that $p(\cdot) = V^* \pi(\cdot) V.$

Now observe, 
\[
\begin{split}
 (\Phi_A\otimes I_K)(\mathcal{Z}) 
 &  = \sum_{\alpha} A_\alpha \otimes \mathcal{Z}_\alpha
  =\sum_{\alpha} A_\alpha \otimes p(\alpha) 
\\ &  =\sum_{\alpha} A_\alpha \otimes V^* \pi(\alpha) V
  = (I\otimes V)^*  A(\pi) (I\otimes V),
\end{split}
\]
where the sums are over $\alpha\in \Her \groupWw\times \Her \semiyN{M}.$
By hypothesis, $A(\pi)\succeq \epsilon.$ Hence
\[
 (\Phi_A\otimes I)(\mathcal{Z}) \succeq \epsilon. \qedhere
\]
\end{proof}

Let $\Pi_{\Her \groupW\times\groupY}$ \index{$\Pi$}
denote the set of unitary representations of $\Her \groupW\times\groupY$ on separable Hilbert space
and set
\begin{equation}
    \label{d:PwM}\index{$\mathcal P$}
 \mathcal{P}_{\Ww,M}^ \epsilon = 
 \{A\in\sAwM :  A(\pi)\succeq \epsilon(I_K \otimes I), \text{ for all } \pi \in \Pi_{\Her \groupW\times\groupY} \}.
\end{equation}

\begin{lemma}
 \label{l:Pcpt}
    The set $\mathcal{P}_{\Ww,M}^ \epsilon$ is compact.
\end{lemma}

\begin{proof}
  That $\mathcal{P}_{\Ww,M}^ \epsilon$ is closed is evident.
  Since it lives in a finite dimensional (normed) vector
  space, it  thus suffices to show $\mathcal{P}_{\Ww,M}^ \epsilon$ is bounded. 
 {To do so} we begin with 
  a  preliminary observation. 

  If $T$ is an operator on Hilbert space and
\[
 g(e^{it}) = I + e^{it} T + e^{-it} T^*\succeq 0
\]
  for all real numbers $t,$ then $\|T\|\le 1.$
 There are many proofs of this fact. Here is one. %
By the operator-valued  Fejér--Riesz theorem \cite{Ro68}, there exist Hilbert 
 space operators $V_0,V_1$ such that 
\[
 g(e^{it}) =(V_0 +e^{it}V_1)^* ( V_0 + e^{it}V_1).
\]
Thus $V_0^* V_0 +V_1^*V_1 = I$ and $V_0^*V_1 =T.$
In particular, $V_0,V_1$ have norm at most one and
hence so does $T.$

  Given a group $G$ with generators $\zeta=\{\zeta_1,\dots,\zeta_\vh\}$,
  let $M$ denote the semigroup generated by $\zeta.$ 
 For $b\in \Her M,$ let $\mbo_b:\Her M\to \CC$ denote the indicator
 function of $b$ and let $\Upsilon_b$ denote the matrix $(\mbo_b(u^{-1}v))_{u,v\in M}.$ 
 In particular, with $e\in G$ the identity, $\Upsilon_e =I.$
 Observe that any given row and column of  $\Upsilon_b$ has at most one non-zero entry, since 
 if $u^{-1}v = b = u^{-1}w,$ then $v=w$ and similarly if $u^{-1}v=b= w^{-1}v,$ then $w=u.$
 It follows that $\Upsilon_b$ determines a bounded operator of norm $1.$  Hence, for all 
 real $t,$
\[
    I + \frac12 \left (e^{it} \Upsilon_b + e^{-it}\Upsilon_b^* \right ) \succeq 0.
\]

 Fix an $A \in \mathcal{P}_{\Ww,M}^ \epsilon.$
 Applying the observation above to $\lralpha =\WYu^{-1} \WYv
\in \Her \groupWw\times \Her \semiyN{M} =\Her (\groupWw\times \semiyN{M})$ (where  $\WYu= \Wu \Ya$ and 
$\WYv=\Wv\Yb$ for some $\Wu,\Wv\in \groupWw$ and $\Ya,\Yb\in \semiyN{M}$) and setting, 
for $t$ real,
\[
 {Z}(t) = I  + \frac12 \left (e^{it} \Upsilon_\alpha + e^{-it}\Upsilon_\alpha^* \right ) \succeq 0,
\]
 we have  $\mathcal{Z}(t)=Z(t)|_{{w}, M} \in \bigcap_W \fC_W.$

 Hence by Lemma~\ref{l:nested}, 
\[
 0 \preceq \Phi_A(\mathcal{Z}(t)) =  I + \frac12 \left ( e^{it} A_{u^{-1}v}   + e^{-it} A_{u^{-1}v}^* \right )
\]
 for all real $t.$  It now follows that $\|A_{u^{-1}v}\|\le 2$ and consequently
$\mathcal{P}_{\Ww,M}^ \epsilon$ is bounded as claimed. 
\end{proof}

Let 
\begin{equation}
\label{d:whC}\index{$\widehat{\mathfrak C}$}
   \wfC{W} =\{\mathcal{Z}\in \fC_W: \exists A \in \mathcal{P}_{\Ww,M}^ \epsilon \; \text{such that} \;  (\Phi_A \otimes I_K)(\mathcal{Z})\not \succ
   \frac{\epsilon}{2}\}.
\end{equation}
 Observe, for $\mathcal{Z}\in \wfC{W}$ and $Z\in \mathscr{L}_{\Ww,W,K}$ such that $\mathcal{Z}=Z\vert_M,$ that $\Phi_A(Z)$ depends only upon $\mathcal{Z}_{g,h} = Z_{g,h}$ for $g,h\in \groupWw\times\semiyN{M}.$

\begin{lemma}
    \label{l:nested:tilde}
With notation and assumption as above,
\begin{enumerate}[\rm (i)]
\item
\label{l:nested:tilde:1}
$\wfC{W} \supseteq \wfC{W+1}$ for each $W\ge M;$ 
\item 
\label{l:nested:tilde:2} 
each  $\wfC{W}$ is  closed and hence compact; and
\item %
\label{l:nested:tilde:3}
\[
 \bigcap_{W=M}^\infty \, \wfC{W}=\varnothing.
\]
\end{enumerate}
Hence, by the finite intersection property, there is an $M'$
such that $\wfC{M'}=\emptyset.$  
\end{lemma}

{
\begin{remark}\rm
 \label{r:nested:tilde}
   We have suppressed the dependence of $\wfC{W}$ on $\epsilon,$
   $K$ and $\Ww$ (equivalently the cardinality of $\groupWw$). 
   The $M^\prime$ in the conclusion depends only on these parameters.
    The conclusion  $\wfC{M^\prime}=\varnothing$ implies that
  $(\Phi(A)\otimes I_K)(\mathcal{Z})\succeq \frac{\epsilon}{2}$
  uniformly over $\mathcal{Z}\in \fC_{M^\prime}$  and polynomials $A$ with $M_K(\CC)$ coefficients
  of bidegree $(\Ww,M)$ that satisfy $A(\pi)\succeq \epsilon.$   
\qed
\end{remark}
}

\begin{proof}[Proof of \ref{l:nested:tilde:1}]
 Immediate.
\end{proof}

\begin{proof}[Proof of \ref{l:nested:tilde:2}]
For notational ease, let $Q=\mathcal{P}_{\Ww,M}^ \epsilon.$
Let $ ({\cZ}_n) $ be a sequence from $\wfC{W}$ and suppose that $({\cZ}_n)$   converges to ${\cZ}$.  
 Since  by Lemma~\ref{l:nested} item~\ref{i:nested:iii},
 $\fC_W$ is compact (hence closed), it follows that ${\cZ} \in \fC_W $. For each ${\cZ}_n$, there exists
 $A_n\in Q$ such that $(\Phi_{A_n} \otimes I_K) ({\cZ}_n) \not \succ \frac{\epsilon}{2}.$ 
  Because $Q$ is compact 
  there exists an $A$ and a subsequence  $(A_{n_l})$ of $(A_n)$ that converges to $A.$
  If $(\Phi_{A} \otimes I_K)({\cZ}) \succ \frac{\epsilon}{2},$ then by joint continuity there is an $\ell$ such that
  $ (\Phi_{A_{n_l}} \otimes I_K)( {\cZ}_{n_l})\succ \frac{\epsilon}{2},$ which is a contradiction.
  Hence $(\Phi_A\otimes I_K)({\cZ}) \not \succ \frac{\epsilon}{2}$ so that ${Z}\in\wfC{W}$ showing
  $\wfC{W}$ is closed. 
\end{proof}

\begin{proof} [Proof of \ref{l:nested:tilde:3}]
If ${\cZ} \in  \bigcap_{W=M}^\infty \, \wfC{W}$, then 
${\cZ} \in  \bigcap_{W=M}^\infty \, \fC_W$.  Thus an application of Lemma~\ref{t:Kpos}
{gives $(\Phi_A\otimes I_K)({\cZ}) \succeq \epsilon$  for all $A \in \mathcal{P}_{\Ww,M}^ \epsilon.$}
But then $\cZ\notin \bigcap_{W=M}^\infty \, \wfC{W}.$ 
Therefore, $\bigcap_{W=M}^\infty \wfC{W}=\emptyset$ as claimed.
\end{proof}

\begin{remark}\rm 
\label{r:no-op-coefficients}
  The proof of the Lemma~\ref{l:nested:tilde} is valid with $K$ replaced by any positive
  integer $K^\prime.$ However, the resulting $M^\prime$  depends upon $K^\prime.$
  Thus the proof given does not produce a single $M^\prime$ independent of $K^\prime$.
  It is for this reason that Theorem~\ref{t:main} is stated for matrix-valued polynomials
  (and not for operator-valued polynomials).  We offer two perspectives on the difficulty,
  even for a single $A$ with operator coefficients normalized so that $A_{e,e}=I.$ 
  First, Proposition~\ref{no-eps-no-problem}
  fails\footnote{Certainly its proof does.} in this case.  To see why, choose a sequence $A_n$ of matrix-valued polynomials satisfying
  the normalization with corresponding (optimal) $\epsilon_n>0$ tending to $0$ and let $A=\oplus A_n.$
  Second, the continuity argument of item~\ref{l:nested:tilde:2} of Lemma~\ref{l:nested:tilde}
  is problematic. At best it produces states $\rho_n$ and ${\cZ}_n$  each of which converge
  in a weak sense 
   and together satisfy $\rho_n(\Phi_A({Z}_n))\le \frac{\epsilon}{2},$ 
   but unfortunately lack of joint continuity prevents the conclusion that
   $\rho_n(\Phi_A({Z}_n))$ converges with limit also at most $\frac{\epsilon}{2}.$\looseness=-1

  The case of the free semigroup $\llx_1$ ($\vg=1$),  $\Her \llx_1 = \mathbb{Z},$ 
  and  $\groupY=\mathbb{Z}^\vh$ (and for polynomials with operator-valued coefficeints)
  is the setting of the factorization results in \cite{Dritschel}.
  In this case, additional structure provided by the fact that positivity of $A$ is 
  certified by positivity of a multi-variate version of a unilateral Toeplitz operator
  (structure that is not available in general) 
  along with a clever construction combined with an approximation argument prevails
  to produce a version of Theorem~\ref{t:main} for strictly psd operator-valued trigonometric
  polynomials defined on the $\vh+1$ torus. 
\qed 
\end{remark}

\section{Proofs of the main results}\label{sec:proofs}

Combining the completion results of  Section \ref{s:freezedtwo} with the realization of Section \ref{sec:psd_functions_on_groups_arise_from_unitary_representations} and the uniform truncation from Section \ref{s:nested-sequences}, we 
show $\groupW\times\groupY$ supports Fej\'er--Riesz factorization with optimal $\groupW$-degree under uniform strict positivity
obtaining a generalization of Theorem~\ref{t:main}.  To conclude this section, we then indicate modifications
of the proof of Theorem~\ref{t:main} in the case $\groupY$ is either trivial or finite that establish the ``perfect'' group-algebra Positivstellensatz on $\Ztg$ of Theorem~\ref{t:main:noY} and the claim made in Remark~\ref{r:main:Y}.

\subsection{Proof of Theorem~\ref{t:main}}
A generalization of Theorem~\ref{t:main} along the lines of Remark~\ref{r:t:main:bounds}
is established in this subsection.

Recall the definition of $\mathcal{P}_{\Ww,M}^\epsilon$  from equation~\eqref{d:PwM} and the definition
of $\Phi_A$ from equation~\eqref{d:PhisubAW} and note that $\mathscr{S}_\Ww \otimes \mathscr{T}_{W}$
is a self-adjoint subspace containing the identity of the $C$-star algebra of 
matrices indexed by the finite set $\groupWw\times \semiyN{W}.$ 
Thus each $\mathscr{S}_\Ww \otimes \mathscr{T}_{W}$ is an operator
system.

\begin{theorem}
\label{t:main:cp}
Let $\Ww\in \groupW$ and a positive integer $M$ be given.
For each $\epsilon>0$ there is a  positive integer  $W\ge M$
such that 
if $A \in \mathcal{P}_{\Ww,M}^\epsilon,$  then 
 $\Phi_A: \mathscr{S}_\Ww \otimes \mathscr{T}_{W} \xrightarrow{} M_K(\CC) $ is completely positive.
\end{theorem}

\begin{proof}
By Lemma~\ref{l:nested:tilde} item~\ref{l:nested:tilde:3}, there is an $W\ge M$ such that $\wfC{W}=\emptyset,$
where $\wfC{W}$ is defined in equation~\eqref{d:whC}.  Thus, given  $A\in \mathcal{P}_{\Ww,M}^\epsilon,$ %
for each $\mathcal{Z}\in \fC_{W},$ we have $\Phi_A(\mathcal{Z}) \succeq \frac{\epsilon}{2}.$

Now  suppose $\mathcal{X} \in \mathscr{L}_{\Ww,W,K} = \mathscr{S}_\Ww \otimes \mathscr{T}_{W}$ and $\mathcal{X} \succ 0$. Because $\mathcal{X} \succ0$, it follows that $\mathcal{X}_{e,e} \succ 0$ as well; hence, $\mathcal{X}_{e,e}$ is invertible. Let $\PXee = I\otimes \mathcal{X}_{e,e}^{-1/2}$. Set 
\[
 \mathcal{X}' =  \PXee \mathcal{X} \PXee. %
\]
Because 
$\mathcal{X}^\prime_{e,e} = I_K$, it follows that  
  $\mathcal{X}^\prime  \vert_M \in {\fC}_{W}$. Thus $\Phi_A(\mathcal{X}^\prime)\succeq \frac{\epsilon}{2}$
  (where we have written $\Phi_A$ in place of $\Phi_A\otimes I_K$ as is customary). 
It now follows that 
\[
\begin{split}
 0 & \preceq (I_K \otimes \PXee^{-1} ) \,  \Phi_A(  \mathcal{X}') \,  (I_K \otimes \PXee^{-1} ) \\ 
   & =  \Phi_A (  \PXee^{-1} \mathcal{X}' \PXee^{-1})  = \Phi_A(\mathcal{X}).
\end{split}
\]

If
$\mathcal{X}\in\mathscr{L}_{\Ww,W,K}^+$  is not strictly positive definite,  {then,}
  by considering $\mathcal{X} + \delta I\otimes I_K$ for $\delta>0,$ a  limiting argument gives $\Phi_A(\mathcal{X}) \succeq 0.$  Thus
 $\Phi_A$ is $K$-positive. 
 Since $\Phi_A$ maps an operator system into $M_K(\CC)$
 is $K$-positive, it is completely positive
 by \cite[Theorem 6.1]{Paulsen}. 
\end{proof}

{
 Given a polynomial $A$ as in the statement of Theorem~\ref{t:main}
 from the introduction,
 an application of Proposition~\ref{no-eps-no-problem} implies $A$ lies in 
 $\mathcal{P}_{\Ww,M}^ \epsilon$ for some $\epsilon>0$ (see
 equation~\ref{d:PwM}). Thus
 the following result is a generalization of Theorem~\ref{t:main}. 
}

\begin{theorem}
\label{t:main:again}
For each $\epsilon>0$ there exists a positive integer $W\ge M$ such that
if $A\in\mathcal{P}_{\Ww,M}^\epsilon,$ then 
there is  an analytic polynomial $B$
of bidegree at most $(\Ww,W)$ such that 
\[
 A= B^*B.
\]
\end{theorem}

\begin{proof}
By Theorem~\ref{t:main:cp}, there exists a $W\ge M$ such that 
$\Phi_A: \mathscr{L}_{\Ww,W,1}\to M_K(\CC)$ is completely positive.
 Let $\mathbb{M}$ denote the space of matrices indexed by $\groupWw\times \semiyN{W}.$ 
 Thus $\mathscr{L}_{\Ww,W,1}$ is naturally identified as a unital self-adjoint subspace
 of $\mathbb M.$  Given ${\WYu},{\WYv}\in \groupWw\times\semiyN{W},$ let $E_{{\WYu},{\WYv}}$ denote the matrix
 with a $1$ in the $({\WYu},{\WYv})$ entry and $0$ elsewhere. 
 Since $\Phi_A$ is completely positive, there exists,
 by the  Arveson extension theorem \cite[Theorem 7.5]{Paulsen},  a completely positive extension $\Psi:\mathbb{M}\to M_K(\CC)$ 
 of $\Phi_A.$  
 Given ${\lralpha} \in \Her \groupWw \times \Her \semiyN{W},$ observe 
\begin{equation}
    \label{e:main1a}
  \sum \{ \Psi(E_{{\WYu},{\WYv}})  : {\WYu},{\WYv}\in \groupWw \times \semiyN{W}, \ \ {\WYu}^{-1} {\WYv} ={\lralpha}\} = \Phi_A(\mathscr{B}_{\lralpha})=A_{\lralpha},
\end{equation}
 where $\mathscr{B}_{\lralpha}$ is defined in equation~\eqref{d:msBalpha}.

 Since $\Psi$ is completely positive, its Choi matrix \cite[Theorem 3.14]{Paulsen}, 
\[
 C_\Psi = \begin{pmatrix} \Psi(E_{{\WYu},{\WYv}}) \end{pmatrix}_{{\WYu},{\WYv}} \succeq 0,
\]
 is psd. 
 Let $N$ denote the size of $C_\Psi.$ Since $C_\Psi$ is psd with
 block $K\times K$ entries, 
 there exist  $B_{\WYu}\in M_{N,K}(\CC)$ 
  such that
\begin{equation}
    \label{e:main2a}
 B_{\WYu}^* B_{\WYv}  =  \Psi(E_{{\WYu},{\WYv}}),
\end{equation}
for ${\WYu},{\WYv}\in \groupWw\times\semiyN{W}.$ 
Combining equations~\eqref{e:main1a} and \eqref{e:main2a} gives,
\[
  A_{\lralpha} = \sum \{ B_{\WYu}^* B_{\WYv}  : {\WYu},{\WYv}\in \groupWw \times \semiyN{W}, \ \  {\WYu}^{-1}{\WYv}={\lralpha}\}.
\]
Thus $A=B^*B,$ where $B$ is the analytic polynomial (as in equation~\eqref{eq:analPoly})
\[
 B = \sum\{ B_\WYu \, \WYu : \WYu \in \groupWw \times \semiyN{W} \}. \qedhere
\]
\end{proof}

\subsection{Proof of Theorem~\ref{t:main:noY}}
 Only the case where $\Her\groupW =\Ztg$ requires a proof,
 since the $\llx_\vg$ case appears in \cite{Mc01}.
 In the case that $\groupY=\{e\},$ the argument in Section~\ref{s:nested-sequences}
 trivializes: for any choice of positive integer $K,$ a partially defined
 psd matrix on  {$\Her \Ztg$}  with entries from $M_K(\CC)$ extends to a psd matrix defined on 
 all of $\Ztg$ and thus is the compression of a unitary representation of  $\Ztg.$ 
 It follows that  $\Phi_A:\mathscr{S}_\Ww \to \cB(\HE)$ as in the statement of Theorem~\ref{t:main:cp}
 is $K$-positive by simply following the proof of that theorem noting the assumption that
 $A(\pi)\succeq 0$ for all unitary representations of $\Ztg$ suffices.
 Hence $\Phi_A$ is completely positive and the rest of the proof is then the same as that 
 of Theorem~\ref{t:main:again}.

 \subsection{Proof of Remark~\ref{r:main:Y}}
  Note that the results of Section~\ref{s:nested-sequences} are not required.
  Rather $\Phi_A:\mathscr{L}_{\Ww,W,1}\to \cB(\HE)$ is, by assumption,
  completely positive, where $W$ is the cardinality of $\groupY.$ 
  Following the proof of Theorem~\ref{t:main:again} establishes
  the claimed result.
  \qed

\section{Examples}
\label{s:examples}

This section collects examples that demonstrate the sharpness of our results. We construct counterexamples on $\ZZ_2*\ZZ_3$ 
(Example~\ref{eg:noZm})
and $\ZZ_3^{*2}$ (Example~\ref{e:noZ32}) showing that the conclusions of Theorems \ref{l:completeZ2}
and  \ref{t:main:noY}, with
their optimal degree bounds, in the sense of Remark~\ref{r:noZm}, can fail.
In both cases, the argument proceeds by exhibiting a partially defined psd matrix with respect to the relevant group that 
does not have a psd completion in the spirit of Theorem~\ref{l:completeZ2}. {It turns out that, in this context, the failure of a psd matrix to have a psd completion is equivalent
to the failure of factorization with optimal degree bounds via 
a cone separation argument. (See Example~\ref{eg:noZm}.)}

It is well known that a classical scalar-valued psd trigonometric polynomial in two variables does not necessarily
factor with optimal degree bounds. Example~\ref{eg:noT2} gives a proof of this fact as a consequence of the existence
of a partially defined psd two-variable Toeplitz matrix (partially defined psd function relative to $\ZZ^2$) 
that does not have a psd extension (to $\ZZ^2$). Compare with Theorems~\ref{l:compHX}.

\begin{example}\rm
\label{eg:noZm}
Let $G$ denote the free group on $x_1,x_2$ modulo the relations $x_1^3=1=x_2^2.$  That is, $G=\ZZ_3 * \ZZ_2$.
Give $G$ the usual shortlex order and let $\Ww=x_2.$ The immediate successor to $\Ww$ is $\Ws=x_1^2.$
A partially defined matrix with respect to
$\Ws=x_1^2$ takes the form
- keeping in mind $x_1^2=x_1^{-1}$ and $x_2=x_2^{-1}$,
\[
 \begin{pmatrix} 1 & x_1 & x_2 & x_1^2 \\
   x_1^2 & 1 & x_1^2 x_2 & x_1 \\
   x_2 & x_2 x_1 & 1 & x_2 x_1^2 \\
   x_1 & x_1^2 & x_1x_2 & 1
 \end{pmatrix}.
\]
 That is, with $I_\Ws=\{\Wu\in G: \Wu\le \Ws\}$ and $J_\Ws=\Her I_\Ws =\{\Wu^{-1}\Wv : \Wu,\Wv\in I_\Ws\}$
 a function $\rho: J_\Ws\to\CC$ corresponds to the matrix,
\[
  \Upsilon_\rho =\begin{pmatrix} \rho(e) & \rho(x_1) & \rho(x_2) & \rho(x_1^2) \\
   \rho(x_1^2) & \rho(e) & \rho(x_1^2 x_2) & \rho(x_1) \\
   \rho(x_2) & \rho(x_2 x_1) & \rho(e) & \rho(x_2  x_1^2) \\
   \rho(x_1) & \rho(x_1^2) & \rho(x_1x_2) & \rho(e)
 \end{pmatrix}.
\]

  Let 
\[
 \mathscr{S} =
 \bigg \{ \begin{pmatrix} \sigma(\Wu^{-1}\Wv) \end{pmatrix}: \sigma: J_w\to \CC\bigg \} 
 = \bigg \{\begin{pmatrix} a_{1,1} & a_{1,2} & a_{1,3}\\ a_{2,1} & a_{1,1} & a_{2,3}
 \\ a_{3,1} & a_{3,2} & a_{1,1} \end{pmatrix} : a_{j,k}\in \CC\bigg \}  \subseteq M_3(\CC),
\]
 and note $\mathscr{S}$ is a (unital) operator system. 
 Further, given $\Wu\in J_\Ww,$
 letting $\mbo_\Wu:J_\Ww\to \CC$ denote the indicator function of $\Wu,$ the matrices
\[
  \Upsilon_{\Wu} = \Upsilon_{\mbo_\Wu}
\]
 form a basis for $\mathscr{S}.$
 
 Consider $\tau:J_{\Ww}\to \CC$ given by 
  $\tau(e)=1;$ $\tau(x_1)=\tau(x_1^2)=\tau(x_1^2x_2)=\tau(x_2x_1)=-
\frac23;$ and $\tau(x_2)=1$  so that 
\[
 \Upsilon_\tau = \begin{pmatrix}  1 & -\frac{2}{3} & 1 \\[1mm]
 -\frac{2}{3} & 1 & -\frac{2}{3} \\[1mm]
 1 & -\frac{2}{3} & 1\end{pmatrix} \in \mathscr{S}, 
\]
 which is evidently psd. 
Now suppose $\rho:J_{\Ws}\to \CC$ extends $\tau.$ 
Thus
\[
  \Upsilon_\rho=\begin{pmatrix} 1 & -\frac{2}{3} & 1 & -\frac{2}{3} \\[1mm]
 -\frac{2}{3} & 1 & -\frac{2}{3} & -\frac{2}{3} \\[1mm]
 1 & -\frac{2}{3} & 1 & z \\[1mm]
 -\frac{2}{3} & -\frac{2}{3} & \bar z & 1
 \end{pmatrix} 
\]
for some choice of $z\in \CC.$  Since the determinant of 
the submatrix of  $\Upsilon_\rho$ spanned by the first, second, fourth rows and columns
 is negative, it is not possible to extend
$\tau$ to a psd function on $J_\Ws;$ 
i.e., the conclusion of Theorem \ref{l:completeZ2} fails for $\ZZ_3*\ZZ_2$.

We next show that the conclusion of Theorem \ref{t:main:noY} also fails for $\ZZ_3*\ZZ_2$.
 It is a special case of a well-known result 
 (see \cite[Theorem~4.8]{Paulsen} and Proposition~\ref{p:freeY})
 that if $p:\ZZ_3* \ZZ_2\to \CC$ is a psd function, then  there is a Hilbert space $\HE$, a unitary representation 
 $\pi:\ZZ_3 * \ZZ_2\to B(\HE)$ and a vector $\he\in \HE$ such that
\[
 p(g) = \langle \pi(g) \he, \he \rangle, 
\]
 for all $g\in \ZZ_3* \ZZ_2.$  
 Thus the set of $\mathscr{P}^+$ of psd partially defined matrices with respect to $\Ww$ 
 that extend to a psd matrix on all of $\ZZ_3 * \ZZ_2$ is in one-one correspondence with the
 partially defined matrices with respect to $\Ww$ that arise from unitary representations
 of $\ZZ_3 * \ZZ_2$ as above. A routine argument (see Section~\ref{s:nested-sequences})
 shows $\mathscr{P}^+$ is closed. What is shown above is that $\Upsilon_\tau$ is in  $\mathscr{S}^+$ (is psd)
 but is not in the closed convex set $\mathscr{P}^+.$ Hence, by Hahn-Banach separation, 
 there is a linear functional $\lambda:\mathscr{S}\to\CC$ such that $\lambda(\Upsilon_\tau)<0$
 and $\lambda(\mathscr{P}^+) \geq 0.$  Setting
\[
 f_\Wu = \lambda(\Upsilon_\Wu),
\]
 for $\Wu\in J_\Ww,$ 
 it follows that the  {trigonometric polynomial} (with scalar coefficients),
\[
  f(x_1,x_2) = \sum_{\Wu \in J_\Ww} f_\Wu \Wu \in \CC[\ZZ_3*\ZZ_2]
\]
 satisfies  $f(U_1,U_2)\succeq 0$ for all pairs of unitary operators $(U_1,U_2)$ satisfying
 $U_1^3=I=U_2^2,$ but
\[
 0> \lambda(\Upsilon_\tau) = \sum_{\Wu\in J_\Ww} \tau(\Wu) f_\Wu.
\]
 On the other hand, if there is a (with possibly vector coefficients)
 $q=\sum_{\Wv\in I_\Ww} q_\Wv \Wv $ such that $f=q^*q,$ then
\[
\lambda(\Upsilon_\tau) =  \sum_{\Wu\in J_\Ww} \tau(\Wu) f_\Wu = \sum_\Wu \left [\sum_{\substack{a,b\in I_\Ww\\ b^{-1}a=\Wu}} q_b^* q_a \right] \, f_\Wu
 = \operatorname{trace}(Q\Upsilon_\tau),
\]
where 
\[
 Q  = \begin{pmatrix} q_b^* q_a \end{pmatrix}_{a,b\in I_\Ww}.
\]
 Since both $Q$ and $\Upsilon_\tau$ are psd, we obtain the contradiction $\lambda(\Upsilon_\tau) \ge 0.$
  \qed
\end{example}

The following  variant of Example \ref{eg:noZm}
shows that 
the conclusions of 
Theorems \ref{l:completeZ2}
and  \ref{t:main:noY}
fail also for $G=\ZZ_3^{*2}$.

\begin{example}\rm
\label{e:noZ32}
Let $G$ denote the free group on $x_1,x_2$ modulo the two relations $x_1^3=1=x_2^3.$ 
Give $G$ the shortlex order. Let $\Ww=x_2$ and $\Ws=x_1^2$ its
immediate successor. A partially defined matrix with respect to
$\Ws=x_1^2$ takes the form
- keeping in mind $x_j^2=x_j^{-1},$
\[
 \begin{pmatrix} 1 & x_1 & x_2 & x_1^2 \\
   x_1^2 & 1 & x_1^2 x_2 & x_1 \\
   x_2^2 & x_2^2 x_1 & 1 & x_2^2 x_1^2 \\
   x_1 & x_1^2 & x_1x_2 & 1
 \end{pmatrix}.
\]
 That is, with $I_\Ws=\{\Wu\in G: \Wu\le \Ws\}$ and $J_\Ws =\{\Wu^{-1} \Wv : \Wu, \Wv\in I_\Ws\}$
 a function $\rho: J_\Ws\to\CC$ corresponds to the matrix,
\[
  \Upsilon_\rho =\begin{pmatrix} \rho(e) & \rho(x_1) & \rho(x_2) & \rho(x_1^2) \\
   \rho(x_1^2) & \rho(e) & \rho(x_1^2 x_2) & \rho(x_1) \\
   \rho(x_2^2) & \rho(x_2^2 x_1) & \rho(e) & \rho(x_2^2  x_1^2) \\
   \rho(x_1) & \rho(x_1^2) & \rho(x_1x_2) & \rho(e)
 \end{pmatrix}.
\]

  Now let $r=-\sqrt{\frac12}$ and  choose $\tau:J_\Ww\to \CC$  by $\tau(e)=1;$ $\tau(x_1)=r=\tau(x_2);$ 
  and $\tau(x_1^2 x_2)=0=\tau(x_2^2 x_1)$   so that
\[
 \Upsilon_\tau = \begin{pmatrix} 1 & r & r \\ r & 1 & 0 \\ r &0&1\end{pmatrix},
\]
 which is evidently psd. If there exists a $\rho:J_\Ws\to \CC$ extending $\tau$ such that
 $\Upsilon_\rho$ is psd, then $3\times 3$  submatrix of $\Upsilon_\rho$ based on its  first, second
 and fourth rows and columns, 
 which does not depend on the values of $\rho(x_2^2x_1^2)$ and $\rho(x_1x_2),$ 
\[
   \begin{pmatrix} 1&r &r\\ r &1 & r\\ r& r&1\end{pmatrix}
\]
 must be psd.   Since it is not, no such $\rho$ exists. 

 Finally, as in Example \ref{eg:noZm} we conclude that there
 is a {trigonometric polynomial} (with scalar coefficients),
\[
  p(x_1,x_2) = \sum_{\Wu \in J_w} p_\Wu \Wu
\]
 such that $p(X_1,X_2)\succeq 0$ for all pairs of operators $(X_1,X_2)$ satisfying
 $X_j^3=I,$ but there does not exist a polynomial (with possibly matrix coefficients),
 $q=\sum_{v\in I_\Ww} q_\Wv \Wv$ such that $p=q^*q.$ 
  \qed
\end{example}

\begin{example}\rm
\label{eg:noT2}
The pattern for a two-variable Toeplitz matrix is determined by $\alpha =e^{is}$ and $\beta=e^{it}.$
We write $\alpha^*$ in place of $e^{-is}$ etc.
For instance, 
\[
 T=\begin{pmatrix} 1 & \alpha & \beta & \alpha^2 & \alpha \beta & \beta^2 & \alpha^2 \beta & \alpha \beta ^2 \\
    *& 1 &  \alpha^* \beta & \alpha & \beta & \alpha^*\beta^2 & \alpha \beta & \beta^2 \\
    *&*&1 & \alpha^2 \beta^* & \alpha & \beta & \alpha^2 & \alpha \beta \\
    *&*&*&1 & \alpha^*\beta & \alpha^{*2}\beta^2 & \beta & \alpha^*\beta^2 \\
    *&*&*&*&1& \alpha^*\beta & \alpha & \beta \\
    *&*&*&*&*&1& \alpha^{2}\beta^* & \alpha\\
    *&*&*&*&*&*&1& \alpha^* \beta \\
    *&*&*&*&*&*&*&1 \end{pmatrix}
\]
is such a matrix, with 
 the columns  with first entries $\alpha^3$ and $\beta^3$ omitted. 
Setting $\alpha\beta=\beta^2=\frac{1}{\sqrt{2}}$ in the upper $7\times 7$ block
and all other entries to $0$ yields the positive two-variable Toeplitz matrix
\[
T_0 =\begin{pmatrix} 1&0&0&0&r&r\\ *&1&0&0&0&0\\ *&*&1&0&0&0\\ *&*&*&1&0&0\\ *&*&*&*&1&0\\ *&*&*&*&*&1 \end{pmatrix},
\]
where $r=\frac{1}{\sqrt{2}}.$
A partial extension of $T_0$  to an $8\times 8$ matrix $T$ has the form 
\[
 T= \begin{pmatrix} 1&0&0&0&r&r&w&z \\ *&1&0&0&0&0&r&r\\ *&*&1&0&0&0&0&r\\ 
    *&*&*&1&0&0&0&0\\ *&*&*&*&1&0&0&0\\ *&*&*&*&*&1&0&0 \\
    *&*&*&*&*&*&1&0\\ *&*&*&*&*&*&*&1 \end{pmatrix}
\]
for $w,z\in\CC.$  Now the lower $7\times 7$ principal minor of $T$ 
is fully specified, but it is not positive semidefinite. Thus, it is not possible to
complete the matrix $T_0$ to a positive semidefinite infinite two-variable Toeplitz matrix.
Arguing as in Example~\ref{eg:noZm}, it follows that there is a scalar-valued 
two-variable trigonometric polynomial 
\[
 p = \sum_{|j|+|k|\le 2} p_{j,k} e^{i js} e^{i kt}
\]
 that does not factor as $p=q^*q$ for $q$ of the form
\[
  q=\sum\{ q_{j,k}e^{i js} e^{i kt} : 0\le j,k, \ \ j+k \le 2\},
\]
for any choice of compatible vectors $q_{j,k}$.\qed
\end{example}

\noindent  Conflict of interest statement: On behalf of all authors, the corresponding author states that there is no conflict of interest.

\medskip

\noindent Data Availability Statement: This manuscript has no associated data.

\makeatletter
\saved@setaddresses                    %
\makeatother

\newpage

\printindex

\newpage

\addcontentsline{toc}{section}{Contents}
\tableofcontents


\begin{thebibliography}{99}
\itemsep=3pt

\bibitem[AHML88]{sparsity}
 Jim Agler, J. William Helton, Scott McCullough, Leiba Rodman,
 {\it Positive semidefinite matrices with a given sparsity pattern,}
Linear Algebra Appl. 107 (1988), 101--149

\bibitem[AM15]{AM15}
Jim Agler, John E. McCarthy, 
{\it Global holomorphic functions in several noncommuting variables},
Can. J. Math. 67 (2015) 241--285. 

\bibitem[Ar76]{Ar76}
William Arveson, 
{\it An invitation to C$^*$-algebras},
Graduate Texts in Mathematics. 39. New York - Heidelberg - Berlin: Springer-Verlag. X, 106 p. (1976). 

\bibitem[BT07]{BT07}
Mihály Bakonyi, Dan Timotin,
{\it Extensions of positive definite functions on free groups},
J. Funct. Anal. 246 (2007) 31--49.

\bibitem[BMV16]{BMV16}
Joseph A. Ball, Gregory Marx, Victor Vinnikov, 
{\it Noncommutative reproducing kernel Hilbert spaces},
J. Funct. Anal. 271 (2016) 1844--1920. 

\bibitem[BW11]{BW}
Mihály Bakonyi, Hugo J. Woerdeman, 
{\it Matrix completions, moments, and sums of Hermitian squares,}
Princeton University Press, Princeton, NJ, 2011, xii+518 pp.
ISBN: 978-0-691-12889-4


\bibitem[Be64]{Be64}
John S. Bell, 
{\it On the Einstein Podolsky Rosen paradox}, 
Physics, 1 (1964), 195.

\bibitem[BCR98]{BCR98}
Jacek Bochnak, Michel Coste, Marie-Françoise Roy, 
{\it Real algebraic geometry},
Ergebnisse der Mathematik und ihrer Grenzgebiete. 3. Folge. 36. Springer, ix, 430 p. (1998). 

\bibitem[BCPSW14]{brunner}
Nicolas Brunner, Daniel Cavalcanti, Stefano Pironio, Valerio Scarani, Stephanie Wehner,
{\it Bell nonlocality}, 
Rev. Modern Phys. 86 (2014), 419--478.



\bibitem[CS16]{steering}
Daniel Cavalcanti,  Paul Skrzypczyk, 
{\it Quantum steering: a review with focus on semidefinite programming}, Rep. Progr. Phys. 80 (2016), 024001.


\bibitem[CHSH69]{CHSH69}
John F. Clauser, Michael A. Horne, Abner Shimony,  Richard A. Holt, {\it Proposed experiment to
test local hidden variable theories},
Phys. Rev. Lett., 23 (1969), 880.

\bibitem[CJRW89]{CJRW}
Nir Cohen, Charles R. Johnson, Leiba Rodman, Hugo J. Woerdeman,
{\it  Ranks of completions of partial matrices,}
Oper. Theory Adv. Appl., 40
Birkhäuser Verlag, Basel, 1989, 165--185.
ISBN: 3-7643-2307-8

\bibitem[Da25]{Da25}
Kenneth R. Davidson, 
{\it Functional analysis and operator algebras},
CMS/CAIMS Books in Mathematics 13. Cham: Springer, xiv, 797 p. (2025). 


\bibitem[dOHMP09]{dOHMP09}
Mauricio de Oliveira, J. William Helton, Scott McCullough, Mihai Putinar, 
{\it Engineering Systems and
Free Semi-Algebraic Geometry}, in: Emerging Applications of Algebraic Geometry, 17--62, IMA Vol. Math. Appl. 149, Springer, 2009.

\bibitem[Dr04]{Dritschel}
Michael Dritschel, {\it
On factorization of trigonometric polynomials,} 
Integral Equations Operator Theory 49 (2004), no. 1, 11--42.

\bibitem[DR10]{DR10}
Michael Dritschel, James Rovnyak,
{\it The operator Fejér-Riesz theorem}, in:
Operator Theory: Advances and Applications 207 (2010) 223--254.

\bibitem[DW05]{DW05}
Michael A. Dritschel, Hugo J. Woerdeman, 
{\it Outer factorizations in one and several variables},
Trans. Am. Math. Soc. 357  (2005) 4661--4679. 

\bibitem[FKMPRSZ+]{FKMPRSZ+}
Marco Fanizza, Larissa Kroell, Arthur Mehta, Connor Paddock, Denis Rochette, William Slofstra, Yuming Zhao,
{\it The NPA hierarchy does not always attain the commuting operator value}, preprint \url{https://arxiv.org/abs/2510.04943}

\bibitem[GW05]{GW05}
Jeffrey S. Geronimo, Hugo J. Woerdeman,
{\it Positive extension, Fejér-Riesz factorization and autoregressive filters in two variables},
Ann. Math. (2) 160 (2005) 839--906.

\bibitem[GKW89]{GKW}
Israel Gohberg, M. (Rien) A. Kaashoek, Hugo J. Woerdeman,
{\it The band method for positive and strictly contractive extension problems: an alternative version and new applications,}
Integral Equations Operator Theory 12 (1989), no. 3, 343--382.

\bibitem[GJSW84]{GJSW}
Robert Grone, Charles R. Johnson, Eduardo M. de  Sá, Henry Wolkowicz, 
{\it Positive definite completions of partial Hermitian matrices,}
Linear Algebra Appl. 58 (1984), 109--124.

\bibitem[Gu16]{Gu16}
Alice Guionnet, 
{\it Free analysis and random matrices},
Jpn. J. Math. (3) 11 (2016) 33--68. 

\bibitem[He02]{He02}
J. William Helton, 
{\it  ``Positive'' noncommutative polynomials are sums of squares},
Ann. Math. (2) 156 (2002) 675--694.


\bibitem[HM04]{HM04}
J. William Helton, Scott McCullough, 
{\it A Positivstellensatz for non-commutative polynomials},
Trans. Amer. Math. Soc. 356 (2004) 3721--3737.

\bibitem[HMP04]{HMP04}
J. William Helton, Scott McCullough, Mihai Putinar,
{\it A non-commutative Positivstellensatz on isometries},
J. Reine Angew. Math. 568 (2004) 71--80. 

\bibitem[JM12]{JM12}
Michael T. Jury, Robert TW Martin, 
{\it Sub-Hardy-Hilbert Spaces in the Non-commutative Unit Row Ball}, in: Function Spaces, Theory and Applications, pp. 349--398. Springer, 2012.

\bibitem[JMS21]{JMS21}
Michael T. Jury, Robert TW Martin, Eli Shamovich, 
{\it Non-commutative rational functions in the full Fock space}, 
Trans. Am. Math. Soc. 374 (2021) 6727--6749. 

\bibitem[KVV14]{KVV14}
Dmitry S. Kaliuzhnyi-Verbovetskyi, Victor Vinnikov, 
{\it Foundations of free noncommutative function theory},
Mathematical Surveys and Monographs 199. American Mathematical Society (AMS), vi, 183 p. (2014). 


\bibitem[KMP22]{KMP22}
 Igor Klep, Victor Magron, Janez Povh,
 {\it Sparse Noncommutative Polynomial Optimization}, Math. Program. 193 (2022) 789--829

\bibitem[KVV17]{KVV17}
Igor Klep, Victor Vinnikov, Jurij Volčič,
{\it Null- and positivstellensätze for rationally resolvable ideals},
Linear Algebra Appl. 527 (2017) 260--293.



\bibitem[Ma08]{Ma08}
Murray Marshall,
{\it Positive polynomials and sums of squares},
American Mathematical Society (AMS), xii, 187 p. (2008). 


\bibitem[Mc88]{2-chordal}
Scott McCullough, {\it   2 -chordal graphs,}
Oper. Theory Adv. Appl., 35
Birkhäuser Verlag, Basel, 1988, 143--192.
ISBN: 3-7643-2221-7


\bibitem[Mc01]{Mc01}
Scott McCullough, {\it 
Factorization of operator-valued polynomials in several non-commuting variables,}
Linear Algebra Appl. 326 (2001), no. 1-3, 193--203.



\bibitem[MSZ+]{MSZ}
Arthur Mehta, William Slofstra, Yuming Zhao,
{\it Positivity is undecidable in tensor products of free algebras}, preprint \url{https://arxiv.org/abs/2312.05617}


\bibitem[MiSp17]{MS17}
James A. Mingo,  Roland Speicher, 
{\it Free probability and random matrices},
Fields Institute Monographs 35. Springer, xiv, 336 p. (2017). 

\bibitem[MuSo11]{MuSo11}
Paul S. Muhly, Baruch Solel, 
{\it Progress in noncommutative function theory},
Sci. China, Math. 54 (2011) 2275--2294. 

\bibitem[NT13]{NT13}
Tim Netzer, Andreas Thom,
{\it Real closed separation theorems and applications to group algebras},
Pac. J. Math. 263 (2013) 435--452.

\bibitem[Nobel22]{nobel}
The Royal Swedish Academy of Sciences,
``Press release: The Nobel Prize in Physics 2022'',
\url{https://www.nobelprize.org/prizes/physics/2022/press-release/}

\bibitem[Oz13]{Oz13}
Narutaka Ozawa, 
{\it About the Connes embedding conjecture},
Jpn. J. Math. (3) 8 (2013) 147--183.



\bibitem[Par78]{parrott}
Stephen Parrott, 
{\it On a quotient norm and the Sz.-Nagy-Foias lifting theorem},
J. Funct. Anal. 30 (1978) 311--325.

\bibitem[PTD22]{PTD22}
James E. Pascoe, Ryan Tully-Doyle, 
{\it The royal road to automatic noncommutative real analyticity, monotonicity, and convexity},
Adv. Math. 407 (2022), Article ID 108548, 24 p. 



\bibitem[Pau02]{Paulsen}
  Vern Paulsen,  {\it Completely bounded maps and operator algebras,}
   Cambridge Stud. Adv. Math., {78}
Cambridge University Press, Cambridge, 2002, xii+300 pp.

\bibitem[Po95]{Po95}
Gelu Popescu, 
{\it Multi-analytic operators on Fock spaces},
Math. Ann. 303 (1995) 31--46. 

\bibitem[Pu93]{Pu93}
Mihai Putinar,
{\it Positive polynomials on compact semi-algebraic sets},
Indiana Univ. Math. J. 42 (1993) 969--984. 


\bibitem[Ro68]{Ro68}
Marvin Rosenblum,
{\it Vectorial Toeplitz operators and the Fej\'er--Riesz theorem},
J. Math. Anal. Appl. 23 (1968) 139--147.


\bibitem[Sc24]{Sc24}
Claus Scheiderer, 
{\it A course in real algebraic geometry. Positivity and sums of squares},
Graduate Texts in Mathematics 303. Springer, xviii, 404 p. (2024). 

\bibitem[SIG98]{SIG98}
Robert E. Skelton, T. Iwasaki, Dimitri E. Grigoriadis,
{\it A unified algebraic approach to linear control
design}, Taylor \& Francis Ltd., 1998.

\bibitem[Sm\"u91]{smu91}
Konrad Schmüdgen, 
{\it The $K$-moment problem for compact semi-algebraic sets},
Math. Ann. 289 (1991) 203--206. 

\bibitem[Smi07]{smith}
Ronald L. Smith, {\it The positive definite completion problem revisited,}
Linear Algebra Appl. {429} (2008) 1442--1452.

\bibitem[Vo10]{Vo10}
Dan-Virgil Voiculescu, 
{\it Free analysis questions. II: The Grassmannian completion and the series expansions at the origin},
J. Reine Angew. Math. 645 (2010) 155--236. 

\end{thebibliography}
\end{document}